%% file: main.tex
\documentclass{article}
\usepackage{graphicx} %
\usepackage{physics}
\usepackage{appendix}
\usepackage{amsmath}
\usepackage{tikz}
\usepackage{mathdots}
\usepackage{mathtools}
\usepackage{yhmath}
\usepackage{cancel}
\usepackage{xcolor}
\usepackage{siunitx}
\usepackage{array}
\usepackage{subcaption}
\usepackage{multirow}
\usepackage{caption}
\usepackage{amssymb}
\usepackage{hyperref}
\usepackage{cleveref}
\usepackage{gensymb}
\usepackage{tabularx}
\usepackage{extarrows}
\usepackage{booktabs}
\usepackage{tikz}
\usepackage{float}
\usetikzlibrary{fadings}
\usetikzlibrary{patterns}
\usetikzlibrary{shadows.blur}
\usetikzlibrary{shapes}

\usepackage{natbib}
\usepackage{cleveref}

\usepackage{amsthm}
\usepackage{amsfonts}
\usepackage{amsmath}
\usepackage{enumitem}

\newtheorem{theorem}{Theorem}
\newtheorem{assum}{Assumption}
\crefname{assum}{assumption}{assumptions}
\newtheorem{lemma}{Lemma}

\newtheorem{definition}{Definition}[section]
\newtheorem{prop}{Proposition}
\crefname{prop}{proposition}{propositions}
\newtheorem{rem}{Remark}
\newtheorem{example}{Example}
\newcommand{\R}{\mathbb{R}}
\newcommand{\cC}{\mathcal{C}}
\newcommand{\N}{\mathbb{N}}

\newcommand{\B}{\mathbb{B}}

\newcommand{\locinnlip}{locally inner Lipschitz continuous}
\newcommand{\Locinnlip}{Local Inner Lipschitz Continuity}

\newcommand{\lilc}{locally inner Lipschitz continuous}

\usepackage{authblk}

\DeclareMathOperator{\BR}{BR}
\DeclareMathOperator{\sgn}{sgn}
\newcommand{\Gr}{\text{Gr}}

\title{On the Connectedness of Sublevel Sets in Invex Optimization}
\author[1,2]{Vinzenz Thoma}
\author[1]{Zebang Shen}
\author[1]{Niao He}
\affil[1]{{Department of Computer Science, ETH Zurich}}
\affil[2]{ETH AI Center}
\affil [ ]{\small{\ttfamily\{vinzenz.thoma, zebang.shen, niao.he\}@inf.ethz.ch}}
\date{}
\tikzset{every picture/.style={line width=0.75pt}} %

\begin{document}

\maketitle

\begin{abstract}
    Understanding the topology of sublevel sets yields crucial insights into the optimization landscape of non-convex functions. If sublevel sets are connected, local search algorithms are less likely to be trapped in isolated valleys, facilitating convergence to global minimizers.
    However, few results exist to establish connectedness in the nonconvex setting. In this work, we present a mathematical toolkit based on the topological mountain pass theorem and use it to study invex functions, a class of functions that includes those satisfying the Polyak-\L{}ojasiewicz inequality and generalizations thereof. We show that their sublevel sets are connected under mild assumptions. We further leverage our result to establish the connectedness of different solution sets for invex-incave minimax problems and incave games.
\end{abstract}

\pagebreak

\pagebreak

\input{chapters/intro}

\input{chapters/related}

\input{chapters/sublevel}

\input{chapters/games}

\input{chapters/conclusion}

\newpage

\bibliography{refs.bib}

\newpage 
\appendix

\include{chapters/appendix}

\end{document}

%% file: chapters/intro.tex
\section{Introduction}
\label{sec:intro}

The geometry of the objective function plays a central role for the success of gradient-based optimization algorithms. A property of particular interest is the \emph{connectedness of sublevel sets}: if every sublevel set $\{x : f(x) \leq c\}$ is connected, there are no ``valleys'' separated by barriers, and gradient-based methods can, in principle, reach global minimizers from any initialization.
This has motivated a growing interest in the topology of nonconvex functions' sublevel sets, especially in the context of overparameterized neural networks \citep{Nguyen2019Connected,Nesterov2022Learning,Zeng2023Connected,Garipov2018Loss,Cooper2021Global,Draxler2018Essentially,Liu2022Loss}.

However, establishing connectedness of sublevel sets in the nonconvex setting is a challenging task. Existing approaches rely on problem-specific, ad-hoc constructions, either by finding a continuous path connecting any two points in the sublevel set \citep{Zeng2023Connected} or by constructing a continuous, surjective function from a convex domain onto the sublevel set \citep{Fatkhullin2021Optimizing}. 

An important class of nonconvex functions are \emph{invex} functions---for which every stationary point is a global minimizer. They thus generalize both convex functions and functions satisfying the famous Polyak-\L{}ojasiewicz (PL) inequality and generalizations thereof~\citep{Polyak1963Gradient,Lojasiewicz1963Une}. Moreover, they capture relevant real-world problems such as machine learning machine learning~\citep{Nishioka2025Revisiting} and reinforcement learning.\footnote{It has been shown that stationary points correspond to global optimal solutions for RL problems with softmax parameterized policies \citep{Mei2020Global} and even for RL problems with general utilities \citep{Zhang2020Variationala}.} 
Invexity has been widely studied in the optimization literature precisely because it guarantees the absence of spurious stationary points.
Yet, surprisingly, the topological structure of the sublevel sets of invex functions remains poorly understood. It is known that sublevel sets of invex functions are lower semi-continuous~\citep{Zang1977functions}. However, without further assumptions the sublevel sets can be disconnected (cf. \Cref{fig:invex}). Identifying further conditions under which the sublevel sets of invex functions are connected is an open question, repeatedly discussed in the literature~\citep{Pini1991Invexity,Pini1994Convexity,Sapkota2021Input,Barik2023Invex}.

In this work, we address this gap by introducing a novel approach for establishing connected sublevel sets, based on the \textit{topological mountain pass theorem}, and use it to show connectedness of the sublevel sets of a broad class of invex functions. These include functions satisfying the Polyak-\L{}ojasiewicz inequality and generalizations thereof.
We further leverage our result to study the connectedness of different solution concepts in invex games and minimax optimization.

%% file: chapters/related.tex
\paragraph{Related Work}

Recently, \citet{Criscitiello2025} showed that the set of minima of a coercive, continuously differentiable function satisfying the global Polyak-\L{}ojasiewicz condition is contractible and therefore connected, which relates to our \Cref{prop:PLconnect} but makes stronger assumptions on the differentiability. During the writing of this paper, we also became aware of the works of \citet{Pucci1984Extensions} and \citet{Silva1998version}, who previously showed that connectedness of the set of minima under certain assumptions on the stationary points of a function can be established via Mountain Pass Theorems. More specifically, \citet{Pucci1984Extensions} showed that if every stationary point of a continuously differentiable function that fulfills the Palais-Smale compactness condition is a local minimum, then every stationary point is a global minimum and the set of minima is connected. Under the same assumptions, \citet{Silva1998version} showed that if $f$ has a stationary point with value $c$, then either the corresponding level set is connected, or there is another stationary point with value $d\neq c$. These works are thus closely related to our \Cref{thm:main}.

%% file: chapters/sublevel.tex
\section{Connected Sublevel Sets of Invex Functions}
\label{sec:sublevel}
We focus on the class of invex functions \citep{Hanson1981sufficiency}, defined by the property that every stationary point is a global minimum. 

\begin{definition}[Invexity]
    We call a differentiable function $f: \R^n \to \mathbb{R}$ \textit{invex} if every stationary point of $f$ is a global minimum. Equivalently, $f$ is invex if there exists a function $\eta: \R^n \times \R^n \to \mathbb{R}^n$ such that for all $x, y\in \R^n$, the following condition holds:
\[ f({y}) \geq f({x}) + \nabla f({x})^T \eta({x}, {y}). \]
\end{definition}

We call a function $f$ \textit{incave}, if $-f$ is invex. Note
invex (resp. incave) functions sharing the same $\eta$ form a convex cone and are closed under composition with monotone increasing convex, (resp. monotone decreasing concave functions) \citep{Mishra2008Invexity}.

While invexity guarantees the absence of spurious local minima, it does not automatically guarantee that the region of global minima is connected (cf. \Cref{fig:invex}). In this section, we identify conditions under which invexity implies the topological connectedness of sublevel sets $L^{-}_c:=\{x\in \R^n| f(x)\leq c\}$.

In the one-dimensional case the connectedness of $L^{-}_c$ is a direct consequence of Rolle's theorem.\footnote{See for example~\citep[Lemma 2.26]{Mishra2008Invexity}.} However, in higher dimensions the sublevel sets of invex functions are not necessarily connected. Consider the counterexample in \Cref{fig:invex}.

    \begin{figure}
        \centering
        \includegraphics[width=0.6\textwidth,,trim=0 0.9cm 0 2.6cm,clip]{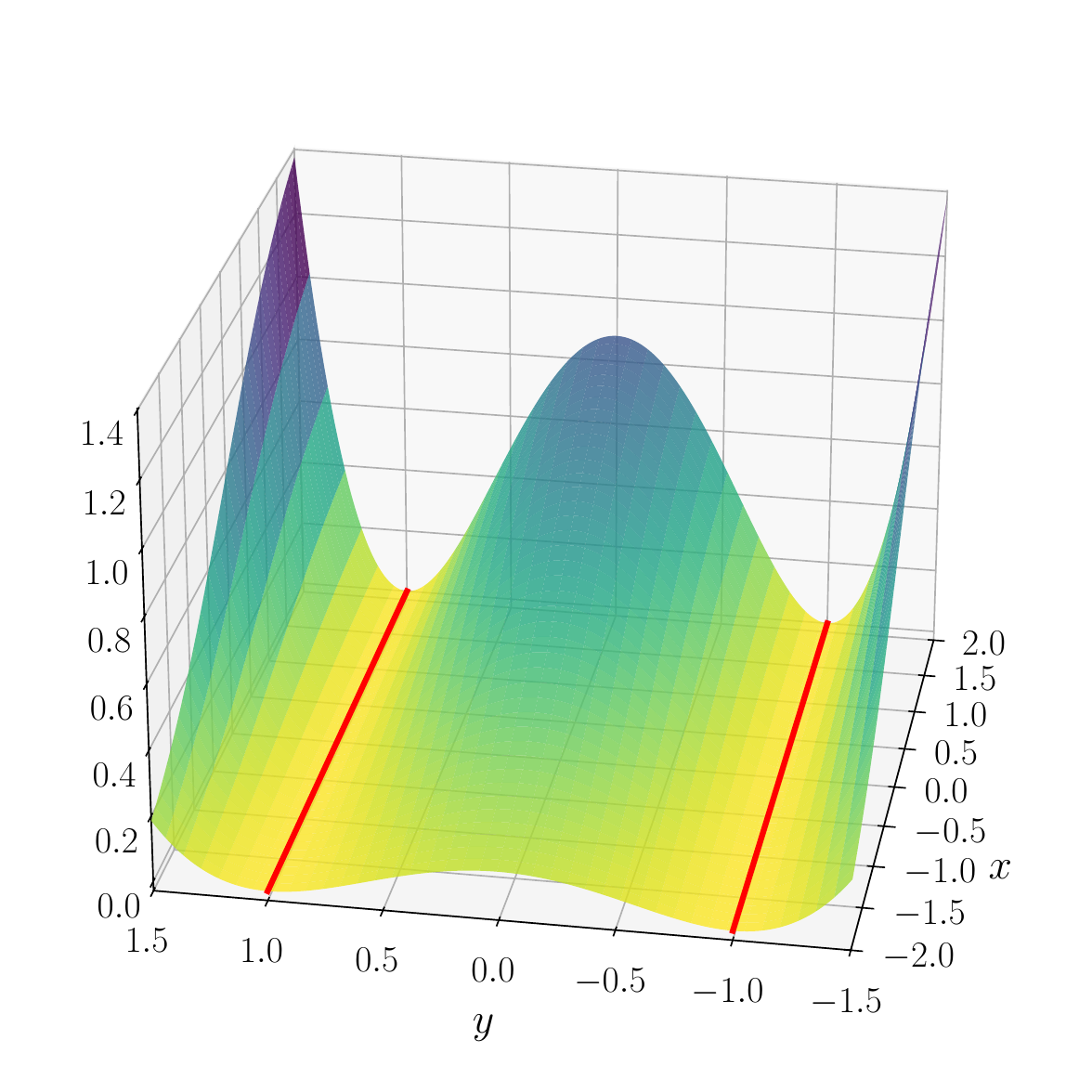}
        \caption{Invexity alone does not imply connected sublevel sets. The function \( f(x,y) = (1/(1+\exp(-x)))(y^2 - 1)^2 \) is invex. However, the set of minima \( \R \times \{-1,1\} \), indicated by the red lines, is not connected.}
        \label{fig:invex}
    \end{figure}

Indeed, to extend the intuition of Rolle's theorem to higher dimensions, we additionally need to assume that the function $f$ is \textit{increasing at infinity}, i.e. for all $x$, there exists a compact set $K$ such that $\forall z \in K^c: f(z)>f(x)$.\footnote{Note that this assumption is weaker than \emph{coercivity}, which requires the function to diverge to infinity. We allow the function to approach a finite asymptote, which allows for a wider range of bounded objectives.} For functions satisfying this property, we can employ the following result by \citet{Katriel1994Mountain}, also referred to as the \emph{topological Mountain Pass Theorem} (MPT).

\begin{theorem}[Topological MPT, \citet{Katriel1994Mountain}]
    \label{thm:katriel}
        Let $f: \R^n \rightarrow \R$ be a continuously differentiable function, increasing at infinity. Let $x_0, x_1 \in \R^n$ and let $\mathrm{S} \subset \R^n$ separate $x_0$ and $x_1$ (that is, $x_0$ and $x_1$ lie in different components of $\R^n\setminus\mathrm{S}$ ), and:
    \[
    \max \left\{f\left(x_0\right), f\left(x_1\right)\right\}<\inf _{x \in \mathrm{S}} f(x)=p.
    \]
    Then there exists a stationary point $x_2$ of $f$ with $f\left(x_2\right)>\max \left\{f\left(x_0\right), f\left(x_1\right)\right\}$.
    \end{theorem}
Building upon this result, we can show that the sublevel sets $L^-_c$ of an invex function $f$ that is increasing at infinity are connected. Intuitively, this holds because if $L^-_c$ were disconnected, we could construct a separating set $S$, which by \Cref{thm:katriel} would yield the existence of another stationary point (the mountain pass) with value larger than $c$, a contradiction to the invexity of $f$.
\begin{theorem}
\label{thm:main}
   Let $f: \R^n \rightarrow \R $ be an invex function with minimum value $f^*$. Further assume that $f\in \cC^1(\R^n)$ and that $f$ is increasing at infinity.
    Then $\forall c\in \R: L^{-}_c$ is connected. In particular the set of minima $\{x\in \R^n|f(x)=f^*\}$ is connected.
\end{theorem}

\begin{figure}
    \centering

\begin{tikzpicture}[x=0.75pt,y=0.75pt,yscale=-1,xscale=1]

\definecolor{myYellow}{HTML}{FDE725FF}   %
\definecolor{myBlue}{HTML}{238A8DFF}     %
\definecolor{myRed}{RGB}{208,2,27}        %
\definecolor{myGreen}{HTML}{73D055FF}      %

\definecolor{myBlue}{HTML}{58C4DD}

\pgfdeclareradialshading{blueShading}{\pgfpoint{0cm}{0cm}}{%
    color(0.7cm)=(myBlue);
    color(1.0cm)=(white)}

\draw [fill=myYellow, draw=myYellow, rounded corners=20pt] (50,70) rectangle (350,230);

\begin{scope}
    \clip (131.16,148.08) circle (49);
    \shade [shading=blueShading] (131.16,148.08) circle (49);
\end{scope}
\draw [color=myYellow] (131.16,148.08) circle (49); %

\begin{scope}
    \clip (262,143) circle (49);
    \shade [shading=blueShading] (262,143) circle (49);
\end{scope}
\draw [color=myYellow] (262,143) circle (49); %

\draw [fill=myRed, fill opacity=0.75, draw opacity=0.64]
    (165,148) -- (141.51,180.34) -- (103.49,167.98) -- (103.49,128.02) -- (141.51,115.66) -- cycle;

\draw [fill=myGreen, fill opacity=1, draw opacity=0.81]
    (275.86,111.95) -- (295.81,146.58) -- (269.04,176.26) -- (232.54,159.97) -- (236.75,120.23) -- cycle;

\fill (123,163) circle (3); %
\fill (249,150) circle (3); %

\node [anchor=north west, inner sep=0.75pt, color=white] at (128,127) {A}; %
\node [anchor=north west, inner sep=0.75pt, color=white] at (258,121) {B}; %
\node [anchor=north west, inner sep=0.75pt, color=black] at (190,105) {S}; %
\node [anchor=north west, inner sep=0.75pt] at (126,150) {$x_{0}$}; %
\node [anchor=north west, inner sep=0.75pt] at (251,136) {$x_{1}$}; %
\node [anchor=north west, inner sep=0.75pt, color=black] at (109,106) {U}; %
\node [anchor=north west, inner sep=0.75pt, color=black] at (247,99) {V}; %

\end{tikzpicture}

    \caption{Construction of the set $\mathrm{S}$ (in yellow) in the proof of Theorem \ref{thm:main}.}
    \label{fig:setSconstruction}
\end{figure}
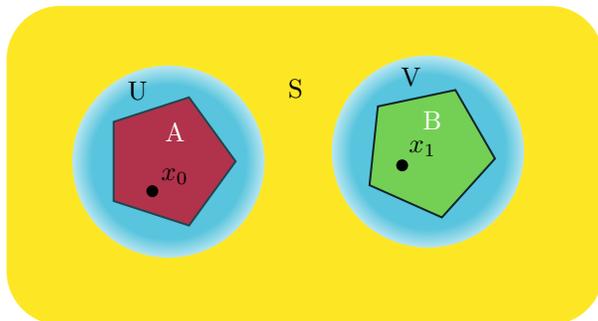

\begin{proof}
    If $c\geq\sup_{x\in \R^n} f(x)$ or $c<f^*$, the statement follows directly. So WLOG let us assume $f^* \leq c<\sup_{x\in \R^n} f(x)$.
    We prove the statement by contradiction and refer the reader to \Cref{fig:setSconstruction} for the visual intuition. Suppose that $L^-_c$ is disconnected. Then there exist disjoint, non-empty, closed sets $A,B$ such that $L^-_c=A \cup B$.

    As $\R^n$ is a normal space, we choose open, disjoint neighborhoods $U,V$ of $A$ and $B$. 
    Using $U$ and $V$, we denote by $\mathrm{S}=\R^n\setminus(U\cup V)$ their complement and choose any two points $x_0 \in A, x_1 \in B$, such that $\max\{f(x_0),f(x_1)\}=c$. This is possible because by definition $A \cup B = L^-_c$. We claim $\mathrm{S}$ fulfills the properties of the separating set in Theorem \ref{thm:katriel} for $x_0,x_1$. 

    First, we need to argue that $\mathrm{S}$ separates $x_0$ and $x_1$. This holds by definition since $\R^n\setminus \mathrm{S}= U \cup V$, $U \cap V =\emptyset$ and $x_0\in U, x_1 \in V.$ 
    
    Second, we need to show that 
    \[\max \left\{f\left(x_0\right), f\left(x_1\right)\right\}= c<\inf _{x \in \mathrm{S}} f(x).\]
    Since $f$ is increasing at infinity, there exists a compact set $K\subseteq \R^n$, such that $\forall x \in K^c : f(x) > c$. It thus suffices to show $c<\inf _{x \in \mathrm{S}\cap K} f(x)$. As $\mathrm{S}\cap K$ is compact, f attains its infimum on $\mathrm{S}\cap K$. As $L^-_c\cap \mathrm{S} \cap K=\emptyset$, it follows that $c<\inf _{x \in \mathrm{S}\cap K} f(x)$.

    Applying Theorem \ref{thm:katriel} to $x_0,x_1$ and $\mathrm{S}$ yields a stationary point $x_2$ with $f(x_2) > \max \left\{f\left(x_0\right), f\left(x_1\right)\right\}=c$. However, by invexity $\nabla f(x_2)=0$ implies $f(x_2)=f^*\leq c$, a contradiction. Therefore, we conclude $L_c^-$ is connected.
\end{proof}

\begin{rem}
    It follows equivalently from \Cref{thm:main} that the superlevel sets $L^+_c=\{x\in \R^n|f(x)\geq c\}$ of an incave function $f$, decreasing at infinity, are connected.
\end{rem}

As discussed previously, invex functions capture several important settings in optimization and machine learning. We therefore briefly discuss how \Cref{thm:main} can be leveraged to derive a novel result on the connectedness of sublevel sets of functions satisfying the Polyak-\L{}ojasiewicz inequality and recover an existing result on connected sublevel sets in the linear quadratic regulator problem.

We begin with the class of functions satisfying the (generalized) Polyak-\L{}ojasiewicz inequality \citep{Polyak1963Gradient,Lojasiewicz1963Une}.

\begin{definition}[$\alpha$-PL condition]
\label{def:aPL}
    A differentiable function \(f:\R^n \to \R \) satisfies the $\alpha$-PL condition if there exists a constant $\mu>0$, such that
    \[
    \|\nabla_x f(x)\|^\alpha \geq \mu \left(f(x) - \min_x f(x)\right).
    \]
\end{definition}
In the case of $f$ being 2-PL and $\mathcal{C}^2(\R^n)$, \citet{Criscitiello2025} showed the set of minimizers is a point. 
Beyond this result, to our knowledge, the question of whether the $\alpha$-PL condition implies connectivity of the sublevel sets remains open \citep{Zeng2023Connected}.  However, as the $\alpha$-PL condition directly implies invexity, we affirmatively answer this open question using \Cref{thm:main}.

\begin{prop}[$\alpha$-PL Implies Connectedness]
    \label{prop:PLconnect}
    Let $f:\R^n \rightarrow \R $ be a continuously differentiable function satisfying the $\alpha$-PL condition for $\alpha>1$. If the set of minima $\{x\in \R^n|f(x)=\min_x f(x)\}$ is bounded then for all $c\in \R$ the sublevel sets $ L^-_c$ are connected, including the set of minima.    
\end{prop}

\begin{proof}
    By definition, the $\alpha$-PL condition implies that $f$ is invex, since every stationary point is a global minimum. Moreover, \Cref{lem:PLgrowth} in \Cref{app:proofs} shows that $\alpha$-PL implies $\frac{\alpha}{\left(\alpha-1\right)}$-growth. If the set of minima is bounded, then this implies $f$ is increasing at infinity and the statement follows from \Cref{thm:main}.
\end{proof}

Similarly, our result can be applied to recover a known result on the connectedness of sublevel sets in the linear quadratic regulator problem.

\begin{rem}
Under the standard assumptions for the state-feedback LQR problem, the objective $f$ on the connected set of stabilizing feedback gains $\mathcal S$ is increasing at infinity and satisfies a non-uniform gradient dominance condition; in particular, $f$ is invex on $\mathcal S$. Hence \Cref{thm:main} recovers the known result of \citet{Fatkhullin2021Optimizing} that every sublevel set of $f$ on $\mathcal S$ is connected.
\end{rem}

%% file: chapters/games.tex
\section{Connected Solution Sets in Nonconvex Games}
\label{sec:connected_equilibria}

Moving beyond standard optimization problems, we examine the landscape of minimax problems and non-cooperative games. In general, solution concepts such as Nash equilibria can be disconnected, complicating their analysis and computation.

In \Cref{sec:zero_sum_games}, we study invex-incave minimax optimization. Leveraging \Cref{thm:main}, we can show that the best response sets are connected, which under certain assumptions implies that the set of saddle points is connected. Next, we focus on $n$-player games. \Cref{sec:potential_games} demonstrates the connectedness of the set of Nash equilibria in incave potential games. For general games with incave utilities, this does not hold. However, in \Cref{sec:incavegames} we prove the connectedness of the set of rationalizable actions, which contains all Nash equilibria.

\subsection{Invex-Incave Minimax Optimization}
\label{sec:zero_sum_games}

In this section, we consider minimax problems and study the connectedness of their solution sets. Minimax optimization captures settings where we optimize over a set of worst-case scenarios. It was originally studied in the context of two-player zero-sum games by \citet{Neumann1928Zur} and has since found applications in nonsmooth optimization, robust optimization and adversarial machine learning. 

Formally, we study the following invex-incave minimax problem, which can be seen as a two-player zero-sum game between a minimizing player choosing $x$ and a maximizing player choosing $y$:
\begin{equation}
    \label{eq:minimax_problem}
    \min_{x\in \R^{n_x}} \max_{y\in \R^{n_y}} f(x,y).
\end{equation}

\begin{assum}
    \label{assum:minmaxinvex}
    Throughout this section, we assume $f \in \mathcal{C}^1(\R^{n_x} \times \R^{n_y})$, that for all $x \in \R^{n_x}$, the function $f(x,\cdot)$ is incave and decreasing at infinity, and that for all $y \in \R^{n_y}$, the function $f(\cdot,y)$ is invex and increasing at infinity. Any deviating or additional assumptions needed are stated at the beginning of the respective result.
\end{assum}

\noindent We use the following notation:
\begin{itemize}
    \item $F(x) := \max_{y\in \R^{n_y}} f(x, y)$ and $G(y) := \min_{x\in \R^{n_x}} f(x,y)$ are the primal and dual functions.
    \item  $X := \{x\in \R^{n_x}: F(x) = \min_{x\in \R^{n_x}} F(x) \}$ and $Y:= \{y\in \R^{n_y}: G(y) =\max_{y\in \R^{n_y}} G(y)\}$ are the primal and dual optimal sets.
    \item $\text{BR}_x(y):=\{x\in \R^{n_x}:x\in \arg\min f(x,y)\}$ and $\text{BR}_y(x):=\{y\in \R^{n_y}:y\in \arg\max f(x,y)\}$ are the best response maps of the $x$-player given $y$, respectively the $y$-player given $x$. 
\end{itemize}

 We are interested in the following three solution concepts and are particularly interested in the topology of the set of saddle points.

\begin{definition}[Solution Concepts for Minimax Optimization]
\label{def:solution}
\noindent
\begin{enumerate}
    \item $(x^*, y^*)$ is a \emph{minimax point}, if for any $(x, y)\in \R^{n_x} \times \R^{n_y}$: $f(x^*, y) \leq f(x^*, y^*) \leq \max_{y'} f(x, y').$
    We denote the set of minimax points by $\underline{M}$.

    \item $(x^*, y^*)$ is a \emph{maximin point}, if for any $(x, y)\in \R^{n_x} \times \R^{n_y}$:
      $\min_{x'} f(x', y) \leq f(x^*, y^*) \leq f(x, y^*).$
    We denote the set of maximin points by $\overline{M}$.

    \item $(x^*, y^*)$ is a \emph{saddle point} or \emph{Nash equilibrium}, if for any $(x, y)\in \R^{n_x} \times \R^{n_y}$:
        $f(x^*, y) \leq f(x^*, y^*) \leq f(x, y^*).$ We denote the set of saddle points by $E= \underline{M} \cap \overline{M}$.  
        If $f$ is invex-incave, all saddle points $(x^*,y^*)$ are stationary points and vice versa, i.e., $ \nabla_x f(x^*, y^*) = \nabla_y f(x^*, y^*) = 0.$
\end{enumerate}

\end{definition}

By \Cref{thm:main}, it follows that $\BR_y(x)$ and $\BR_x(y)$ are connected and compact. However, it is an open question if the connectedness extends to $E,\underline{M}$, and $\overline{M}$. We first show an intermediate result that the union of the minimax and maximin points $\underline{M} \cup \overline{M}$ is connected.

\begin{prop}
\label{prop:stackelberg}
    Under \Cref{assum:minmaxinvex} and assuming that $E\neq \emptyset$,\footnote{This ensures all sets exist. Note an empty set is trivially connected.} the set $\underline{M} \cup \overline{M}$ is connected.
\end{prop}

From \Cref{def:solution}, it follows that $E \subseteq \underline{M}$ and $E \subseteq \overline{M}$. As $E$ is assumed to be non-empty, it follows that the minimax and maximin values coincide, i.e. $\min_x\max_y f(x,y)=\max_y\min_x f(x,y)$. Note, this does not imply that every minimax or maximin point is also a saddle point. However, if $E=\underline{M}=\overline{M}$ holds, then all three sets are equal to $\underline{M} \cup \overline{M}$ and thus connected by \Cref{prop:stackelberg}. We can prove this under the additional assumption of inner Lipschitz continuity of the best response maps.

\begin{definition}[\Locinnlip]
\label{def:lilc}
A set-valued map $S:\R^n \rightrightarrows \R^{m}$ is \locinnlip\ at $\overline{x}$ for $\overline{y}$, if $\overline{y}\in S(\overline{x})$ and there exists a constant $\kappa \geq 0$ and an open neighborhood $N_{\overline{x}}$ of $\overline{x}$ such that:
\begin{equation*}
    \forall x \in N_{\overline{x}}: d(\overline{y},S(x)) \leq \kappa \|x-\overline{x}\|,
\end{equation*}
\end{definition}

\begin{theorem}
\label{thm:lilc} Assume that for all $(x,y)\in \underline{M}$, $\BR_y(x)$ is \lilc\ at $x$ for $y$ and conversely for all $(x,y)\in \overline{M}$, $\BR_x(y)$ is \lilc\ at $y$ for $x$. Then $E=\underline{M}=\overline{M}$ and this common set is connected.
\end{theorem}

The full proof can be found in \Cref{sec:proof:lilc}. The intuition behind it is as follows.
Take a minimax point $(x',y')\in \underline{M}$ for example. If it were not a saddle point, then $\nabla_x f \neq 0$. Moreover, the inner Lipschitz continuity prevents the best response map $\BR_y(x)$ from moving away too far from $y'$, as we vary $x$. Therefore, we can choose a small enough $\delta>0$ such that $F(x'-\delta \nabla_x f(x',y'))<F(x')$, contradicting the assumption that $(x',y')$ is a minimax point. The same argument applies to maximin points. Therefore, all minimax and maximin points are saddle points and thus $E=\underline{M}=\overline{M}$. By \Cref{prop:stackelberg}, this common set is connected.

We are naturally interested in identifying which properties of $f$ guarantee that $\BR_x$ and $\BR_y$ are \lilc. This is because $f$, in contrast to the best response maps, is given in explicit form in \eqref{eq:minimax_problem} and thus conditions on $f$ are generally easier to verify. One important such property is the \emph{two-sided PL condition}, which characterizes a class of nonconvex-nonconcave minimax problems and has attracted recent interest in the optimization community.

\begin{definition}[Two-sided PL inequality \citep{yang2020global}]
    \label{def:2PL}
A continuously differentiable function $f(x, y):R^{n_x}\times \R^{n_y}\to \R$ satisfies the two-sided PL condition if there exist constants $\mu_1, \mu_2>0$ such that:
$$
\begin{aligned}
&\forall (x,y)\in \R^{n_x}\times \R^{n_y}: \left\|\nabla_x f(x, y)\right\|^2 \geq 2 \mu_1\left(f(x, y)-\min _x f(x, y)\right) \\
& \forall (x,y)\in \R^{n_x}\times \R^{n_y}:\left\|\nabla_y f(x, y)\right\|^2 \geq 2 \mu_2\left(\max _y f(x, y)-f(x, y)\right)
\end{aligned}
$$
\end{definition}

\begin{prop}
\label{prop:PLconnected}
Let $f$ satisfy the two-sided PL condition. Further, assume that $\nabla_y f(\cdot,\overline{y})$ is locally $L_{\overline{y}}$-Lipschitz continuous at all $(\overline{x},\overline{y})\in \underline{M}$, that $\nabla_x f(\overline{x},\cdot)$ is locally $L_{\overline{x}}$-Lipschitz continuous at all $(\overline{x},\overline{y})\in \overline{M}$, as well as that for all $x,y$, the set of minima of $f(\cdot,y)$ and maxima of $f(x,\cdot)$ are bounded. Then $E=\underline{M}=\overline{M}$ and this common set is connected.
\end{prop}

\begin{figure}
    \centering
    \includegraphics[width=0.6\textwidth,trim=0 0.6cm 0 1.7cm,clip]{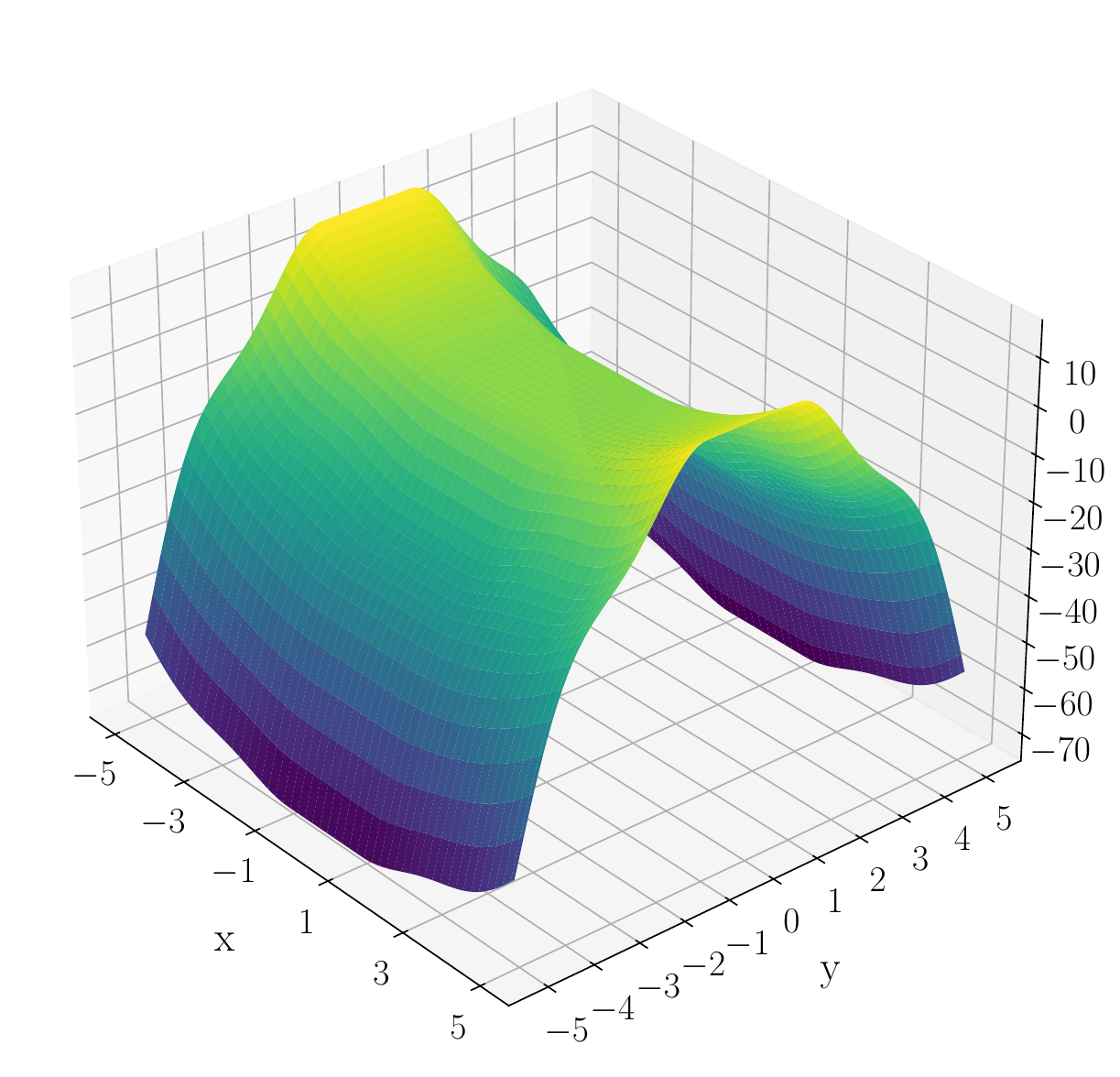}
    \caption{For $x,y \in \R$, define $a = \max\left(|x| - 1, 0\right)$ and $b = \max\left(|y| - 1, 0\right)$. The plot shows the function $f(x,y)= a^2 +3 \sin^2(b)\sin^2(a) - 4b^2-10\sin^2(b)$.
    This function is nonconvex-nonconcave, but satisfies the assumptions of \Cref{prop:PLconnected} (verified in \Cref{sec:proof:ex:2PLfunction}). Therefore $E=\underline{M}=\overline{M}$ and this common set is connected. A direct computation shows $E=[-1,1]\times[-1,1]$.}
    \label{fig:3d_plot_sine}
\end{figure}
An example of a nonconvex-noncave function, satisfying the assumptions of \Cref{prop:PLconnected} and thus with a connected set of saddle points, is given in \Cref{fig:3d_plot_sine}. 
\begin{rem}[Relaxation to Local Inner Hölder Continuity]
    In \Cref{thm:lilc}, we assumed that the best response maps are locally inner Lipschitz continuous. This can be further relaxed to local inner Hölder continuity if we instead assume that $\nabla f$ is locally Lipschitz continuous on $\underline{M}\cup\overline{M}$ and that $f$ locally satisfies the $\alpha$-PL condition. A complete exposition of these results can be found in \Cref{app:Holder}.
\end{rem}

\subsection{Incave Potential Games}
\label{sec:potential_games}
In this section, we study incave potential games and will show that their Nash equilibrium is connected.

We denote an $n$-player game by a tuple $(N,\mathcal{A},u)$, where $N=\{1,\ldots,n\}$ is the set of players, $\mathcal{A}=\prod_{i=1}^n A_i$ is the continuous joint action space and $u=(u_1,\ldots,u_n)$ is the joint utility, where $u_i:\mathcal{A}\to \R$ is the utility function of player $i$. We assume all $A_i = \R^{n_i}$ for some $n_i \in \mathbb{N}$. We adopt the the subscript notation $a_{-i}$ to denote the actions of all players except player $i$ and $a_i$ to denote the action of player $i$. Analogous to \Cref{sec:zero_sum_games}, we denote by $\BR_i(a_{-i}) =\arg\max_{a_i\in A_i} u_i(a_i,a_{-i})$ the best response map of player $i$ to action $a_{-i}$. We begin with the standard definition of a Nash equilibrium. 

\begin{definition}[Nash equilibrium]
An action profile $a=(a_1,\dots,a_n)$ is a \textit{Nash equilibrium} if for all $i\in N$ it holds that $a_i\in \BR_i(a_{-i})$.
\end{definition}

\citet{Monderer1996Potential} introduced continuous potential games, where the utilities of all players are related through a single potential function. In this section, we will study them under the added assumption of incavity.

\begin{definition}[Incave Potential Game]
        A game $(N,\mathcal{A},u)$ is an \emph{incave potential game} if the utility functions $u_i$ are continuously differentiable and there exists a \emph{potential function} $P:\mathcal{A}\to \R$ that satisfies $\frac{\partial P(a)}{\partial a_i} = \frac{\partial u_i(a)}{\partial a_i}$ for all $i\in N$ and $a\in \mathcal{A}$. Moreover, $P$ is continuously differentiable, incave, and decreasing at infinity.

\end{definition}

For potential games, it was shown that the set of Nash equilibria is equal to the set of maxima of the potential function \citep{Monderer1996Potential}. Based on this result, several works have discussed how a (strictly) concave potential function implies the existence of a (unique) Nash equilibrium \citep{Neyman1997Correlated,Mertikopoulos2018Learning}. Using \Cref{thm:main}, we can easily extend these results to the case of incave potential functions.

\begin{prop}
    \label{prop:incave_potential}
    Let $(N,\mathcal{A},u)$ be an incave potential game with potential function $P$. Then the game has a Nash equilibrium and the set of Nash equilibria is connected.
\end{prop}

Below, we present an example of an incave potential game motivated by applications in economics and show how \Cref{prop:incave_potential} applies to it.

\begin{example}
    \label{ex:econincave}
    Consider a game where the utility functions take the following form:
    \begin{equation}
        \label{eq:utility}
        u_i(a) = b(a)-c_i(a_i),
    \end{equation}
    where $b: \mathcal{A} \to \R$ is a shared reward (or cost) function of all players and $c_i: A_i \to \R$ is the individual cost (or reward) of player $i$ for playing $a_i$. This captures, for example, problems of pollution, the tragedy of the commons \citep{Hardin1968Tragedy}, congestion games \citep{Rosenthal1973class}, or team games with a shared reward function and private costs or benefits. The game has a potential function given by \(
        P(a) = b(a) - \sum_{i=1}^n c_i(a_i).
    \)
    Assume that $b$ is incave with respect to an $\eta$ and decreasing at infinity and that all $c_i$ are invex with respect to the same $\eta$ and increasing at infinity. Then $P$, as a linear combination of those functions, is incave with respect to $\eta$ and decreasing at infinity \citep{Mishra2008Invexity}. Therefore, by \Cref{prop:incave_potential}, the set of Nash equilibria is connected.
\end{example}

\subsection{$n$-player Incave Games}
\label{sec:incavegames}

While the previous section established the connectedness of Nash equilibria in incave potential games, this property does not extend to $n$-player incave games. However, we show that for those games, the set of \textit{rationalizable} actions, which contains all Nash equilibria, is connected. 

We begin by formally defining incave games below. They generalize the class of concave games, introduced by \citet{Rosen1965Existence}, where the utility functions are concave in the actions of the corresponding player and which captures important applications in economics. Note, we continue to use the game-theoretic notation introduced in \Cref{sec:potential_games}.

\begin{definition}[Incave Games]
    \label{def:incavegames}
    A game $(N,\mathcal{A},u)$ is called incave, if for all $i\in N$, the utility function $u_i$ is incave in $a_i$ and $u_i$ is decreasing at infinity with respect to $a_i$.
\end{definition}

\begin{figure}
    \centering
    \includegraphics[width=\textwidth,trim=0.3cm 0.3cm 0.3cm 0.2cm,clip]{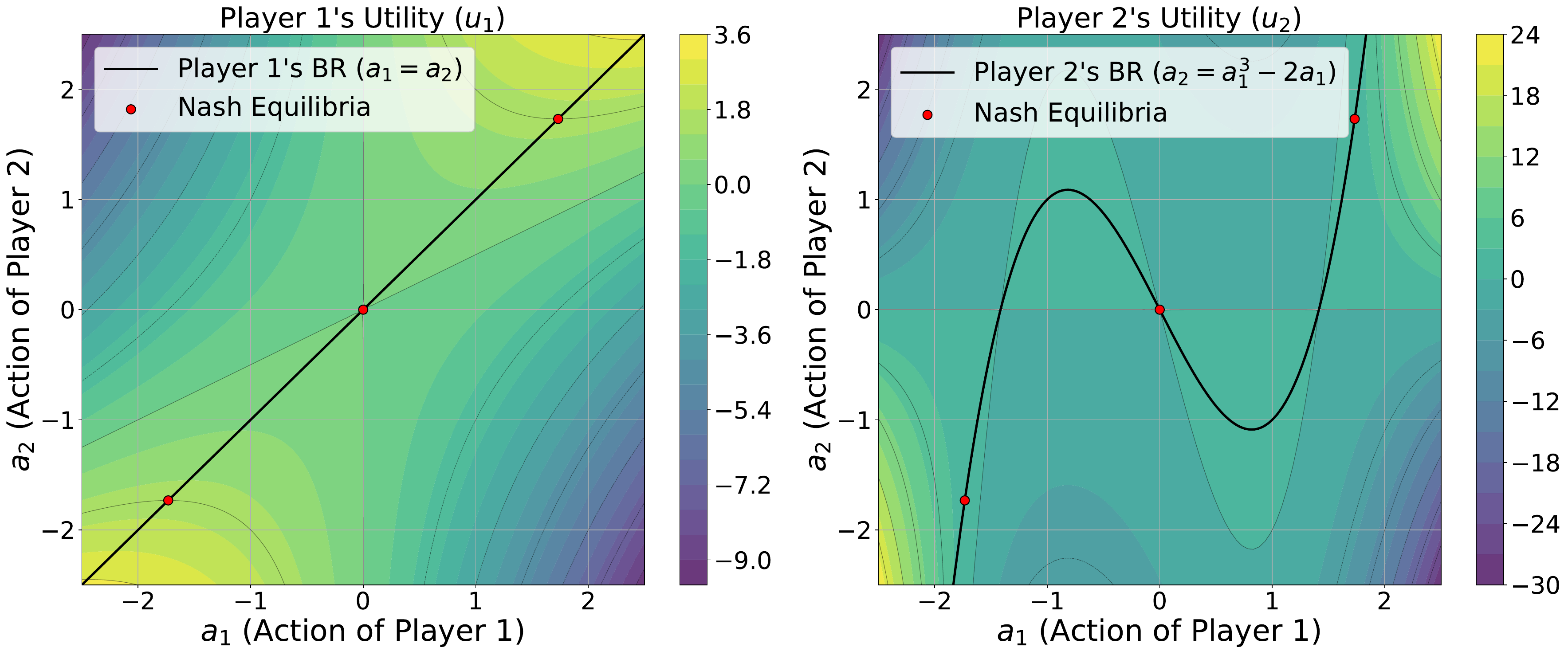}
    \caption{Plot of the utility functions $u_1(a_1,a_2)=-\tfrac{1}{2} a_1^2 +a_1 a_2$ and $u_2(a_1,a_2)=-\tfrac{1}{2} a_2^2 +(a_1^3-2 a_1)$ of a two-player game. The utilities are incave in $a_i$ and have three disconnected Nash equilibria $(0,0)$, $(-\sqrt{3},-\sqrt{3})$ and $(\sqrt{3},\sqrt{3})$.}
    \label{fig:nash_disconnect_utilities}
\end{figure}
In general the set of Nash equilibria is not connected for incave games, not even if we restrict to concave games. This is illustrated in \Cref{fig:nash_disconnect_utilities}.

In this section, we thus study the set of $S,k$-rationalizable and $S$-rationalizable actions. Both constitute important solution concepts for three reasons: (1) they arise naturally from assumptions on players' rationality and common knowledge thereof, (2) $S,k$-rationalizable actions are good predictors of human behavior, and (3) $S$-rationalizable actions are the limit set of best-response dynamics.

We begin by defining the set-valued joint best response map, which maps $S\subset \mathcal{A}$ to a subset of $\mathcal{A}$ and is defined as $\lambda(S)= \prod_{i=1}^n \cup_{a\in S} \BR_i(\pi_{-i}(a))$, where $\pi_{-i}$ is the projection of a joint action $a$ onto the actions of all players except $i$. Intuitively, $\lambda(S)$ is the set of actions, where each player plays some best response to a possible action of the other players in $S$.
We recursively define $\lambda^k(S)=\lambda \circ \lambda^{k-1}(S)$ and call the set $\lambda^k(S)$ the set of $S,k$--rationalizable actions. Note that $\lambda^k(S)$ contains all actions that survive $k$ iterations of simultaneous elimination of strongly dominated actions, i.e. the elimination of all actions $a_i$ that are not a best response to any $a_{-i}\in \lambda^{k-1}(S)$.

The concept of $S,k$--rationalizable actions goes back to \citet{Bernheim1984Rationalizable}. It captures the idea of \emph{bounded rationality}, i.e. that players have limited computational power and can only simulate their own, as well as their opponents' best responses up to a finite number of $k$ levels. It was later verified that humans tend to play $k$-rationalizable actions, even in simple games \citep{Nagel1995Unraveling,Stahl1995Players}.

As we show next, starting from a connected set $S$, $\lambda^k(S)$ preserves connectedness for all $k$ in incave games. The proof of this result crucially relies on \Cref{thm:main}, as well as the outer semicontinuity of the best response maps, which is a consequence of the incavity of each $u_i$ with respect to $a_i$.

\begin{prop}
    \label{prop:connected_k_rationalizable}
    Let $S \subset \mathcal{A}$ be a connected set. Then the set of $S,k$-rationalizable actions $\lambda^k(S)$ is connected for all $k\in \N$.
\end{prop}

Next, we define the set of $S$-\emph{rationalizable actions} $R(S)=\bigcap_{k=1}^\infty \lambda^k(S)$ as the set of actions, surviving arbitrarily many iterations of simultaneous elimination of strongly dominated actions. Rationalizability was independently introduced by \citet{Bernheim1984Rationalizable} and \citet{Pearce1984Rationalizable} and generalizes the concept of a Nash equilibrium.\footnote{In this work, we restrict to the set of so-called \emph{point rationalizable strategies}. We refer to \citet{Bernheim1984Rationalizable} for a full discussion.} Indeed, rationalizability only requires rationality---players want to maximize their utility---and common knowledge of rationality, i.e. players know that their opponents are rational and that they know that they know that their opponents are rational, etc. A Nash equilibrium additionally requires that the beliefs of the players are consistent with actual behavior, i.e. that the other players actually play the action $a_{-i}$, that player $i$ believes and best responds to.

As $\mathcal{A}$ is not compact $R(S)$ may in general be empty. We therefore make the additional assumption of \emph{strategic compactness} to show that $R(S)$ is non-empty, compact and connected.

\begin{definition}[Strategic Compactness]
    \label{def:strategic_compactness}
    An incave game is strategically compact, if there exists a compact set $K\subset \mathcal{A}$, such that $\lambda(K)\subseteq K$. 
\end{definition} 
Given the strategic compactness of an incave game with a compact set $K$, in combination with the connectedness of $\lambda^k(K)$ for all $k\in \N$, it follows that $R(K)$ is connected.
\begin{prop}
    \label{prop:connected_rationalizable}
    Consider an incave game that is strategically compact for a set $K$. Then $R(K)$ is non-empty, compact and connected.
\end{prop}

%% file: chapters/conclusion.tex
\section{Conclusion}
In this work, we illustrated that the topological mountain pass theorem is a powerful tool to study the optimization landscape of nonconvex functions. We proved that for invex functions satisfying a mild growth condition, namely increasing at infinity, all sublevel sets are connected. In particular, this covers the practically important class of functions satisfying the (generalized) Polyak-\L{}ojasiewicz inequality. We further explored how invexity leads to various connected solution concepts in different types of games.

Understanding the topological landscape of functions is an important step towards the design of novel and performant algorithms as illustrated in the recent works of \citet{Masiha2025Superquantile,Gong2025Poincare}. We thus believe future research can build upon our results to tackle various nonconvex problems and design algorithms that explicitly exploit this connectivity.

%% file: chapters/appendix.tex
\section{Proofs}
\label{app:proofs}

This appendix provides the remaining proofs. We first prove \Cref{thm:lilc} in \Cref{sec:proof:lilc}. We then establish the proposition-level results in the order they are used: the connectedness of $\underline{M}\cup \overline{M}$ in \Cref{sec:proof:stackelberg}, connectedness under the two-sided PL condition in \Cref{sec:proof:PLconnected}, connectedness in incave potential games in \Cref{sec:proof:incave_potential}, and connectedness of rationalizable sets in \Cref{sec:proof:connected_k_rationalizable,sec:proof:connected_rationalizable}. We also verify the examples from \Cref{fig:3d_plot_sine,ex:localPLconnected} in \Cref{sec:proof:ex:2PLfunction,sec:proof:ex:localPLconnected}.

\Cref{sec:auxiliaryresults} contains the statements and proofs of several auxiliary lemmas. Namely, \Cref{lem:addeq} on the interchangeability of Nash equilibria, \Cref{prop:eqproduct} on the product structure of $E$, \Cref{prop:EB} on connectedness under the EB property, \Cref{lem:uniformcont}, which is a variation of the Heine-Cantor theorem, \Cref{lem:osc} and \Cref{lem:isc} establishing outer and inner semicontinuity of best response maps, \Cref{lem:connectedBRgraph} on the connectedness of the best response graph, \Cref{lem:finiteintersection} that establishes the connected finite intersection property and \Cref{lem:PLgrowth} that shows $\alpha$-PL implies $\frac{\alpha}{\alpha-1}$ growth.

Appendix~\ref{app:Holder} then complements these proofs by showing how \Cref{thm:lilc} extends from local inner Lipschitz continuity to local inner Hölder continuity under stronger regularity of $f$.

\subsection{Proof of \Cref{thm:lilc}}
\label{sec:proof:lilc}
\begin{proof}
    WLOG, we focus on $\BR_y(x)$ being \lilc\ to show that $\underline{M}\subseteq E$. The converse, i.e. $E\subseteq \underline{M}$ holds by \Cref{def:solution}. The proof for $\BR_x(y)$ and $\overline{M}$ follows analogously, which in combination with \Cref{prop:stackelberg} proves $E$ is connected.

    Suppose for contradiction there exists $(x,y)\in \underline{M}\setminus E$.
    Since $(x,y)\in \underline{M}\setminus E$, it follows by the invex-incave structure of $f$ that $\nabla_y f(x,y)=0$, but $\nabla_x f(x,y)\neq 0$.
    By assumption, $\BR_y(x)$ is \lilc\ at $x$ for $y$ in some neighborhood $N_x$.
    Consider $\delta'$ small enough, such that for all $\delta \in [0,\delta']$, $x_\delta:=x- \delta \nabla_x f(x,y) \in N_x$. Given that $\BR_y(x_{\delta})$ is compact, there exists a point $y_{\delta} \in \BR_y(x_\delta)$, such that $\|y-y_{\delta}\|=d (y,\BR_y(x_{\delta}))$. In particular, using the local inner Lipschitz continuity and the definition of $x_\delta$, it holds that:
    \begin{equation*}
       \|y-y_{\delta}\| =d (y,\BR_y(x_{\delta})) \leq \kappa \|x-x_{\delta}\| = \kappa \delta \|\nabla_x f(x,y)\|.
    \end{equation*} 

    This leads us to the following inequality for all $\delta \in [0,\delta']$:
    \begin{align*}
    F(x_\delta) = &f(x_\delta,y_\delta) \\
        =& f(x,y) + \nabla_x f(x,y)^T(x_\delta-x) + \nabla_y f(x,y)^T(y_\delta-y) \\
        &+ o(\|x_\delta-x\| + \|y_\delta-y\|)\\
        =& f(x,y) + \nabla_x f(x,y)^T(x_\delta-x) + o(\|x_\delta-x\| )\\ 
        =& F(x) - \delta \|\nabla_x f(x,y)\|^2 + o(\delta \|\nabla_x f(x,y)\|),
    \end{align*} 
    where the first equality follows from the definition of $F$ and the fact that $y_\delta\in \BR_y(x_\delta)$, the second equality follows from the differentiability of $f$, the third equality follows from the invexity of $f(x,\cdot)$ and the fact that $y\in \BR_y(x)$ and the local Lipschitz continuity of $\BR_y$, and the last equality uses that $(x,y)\in \underline{M}$ (and thus $F(x)=f(x,y)$) , and the definition of $x_\delta$.
    By Assumption $\nabla_x f(x,y)\neq 0$, which implies for small enough $\delta$, $F(x_\delta)<F(x)$, a contradiction to the supposition that $x\in X$. Therefore, $\underline{M}\subseteq E$ and $E$ is connected.
\end{proof}

\subsection{Proof of \Cref{prop:stackelberg}}

\label{sec:proof:stackelberg}
\begin{proof}

   As $E\neq \emptyset$, let $(x^*,y^*)\in E$. Moreover, since $E=\underline{M}\cap \overline{M}$, we know that $\underline{M}\cup\overline{M}$ is non-empty and contains $(x^*,y^*)$. We will now show that any point in $\underline{M}\cup\overline{M}$ lies in the same connected component as $(x^*,y^*)$, which implies $\underline{M}\cup\overline{M}$ is connected.

   Take any point $(x,y)\in \underline{M}\cup\overline{M}$ and WLOG assume $(x,y)\in \underline{M}$ (the proof works analogously for $\overline{M}$). By \Cref{prop:eqproduct}, it follows that $(x,y^*)\in E$ and that $x$ and $x^*$ both belong to $\BR_x(y^*)$. We first claim that $(x,y^*)$ is in the same connected component as $(x^*,y^*)$. Assume this were not the case, then there exist disjoint, nonempty, open sets $A,B$, such that $(x,y^*)\in A$, $(x^*,y^*)\in B$, and $\underline{M}\cup\overline{M}=A\cup B$. However, we know that $y^*\in Y$ and thus $\BR_x(y^*)\times \{y^*\} \subset \underline{M}\cup\overline{M}$. Therefore, it follows that $A_{y^*}=A \cap \left(\BR_x(y^*)\times \{y^*\}\right)$ and $B_{y^*}=B \cap \left(\BR_x(y^*)\times \{y^*\}\right)$ are disjoint, nonempty, open sets whose union is $\BR_x(y^*)\times \{y^*\}$. However, the latter is a connected set by \Cref{thm:main}, a contradiction.

   Next, we show that $(x,y)$ is in the same connected component as $(x,y^*)$. The proof works analogously. We know that $x\in X$ and thus $\{x\} \times \BR_y(x)\subset \underline{M}\cup\overline{M}$. If $(x,y)$ were not in the same connected component as $(x,y^*)$, then there exist disjoint, nonempty, open sets $A,B$, such that $(x,y)\in A$ and $(x,y^*)\in B$ and we can define disjoint, nonempty, open sets $A_{x}=A \cap \left(\{x\} \times \BR_y(x)\right)$ and $B_{x}=B \cap \left(\{x\} \times \BR_y(x)\right)$. However, $\{x\} \times \BR_y(x)$ is connected, a contradiction.
\end{proof}

\subsection{Proof of \Cref{prop:PLconnected}}
\label{sec:proof:PLconnected}
\begin{proof}
    From \Cref{def:2PL}, it follows that $f$ satisfying the two-sided PL function implies that it is invex-incave. Moreover, \Cref{lem:PLgrowth} in \Cref{app:proofs} in combination with the bounded sets of minima and maxima implies that $f(\cdot,y)$ is increasing at infinity for all $y$ and that $f(x,\cdot)$ is decreasing at infinity for all $x$. Next, we show that $f$ globally satisfies the EB property in $x$. The proof works equivalently for the $y$ variable. 
    Combining the PL condition with the result from \Cref{lem:PLgrowth}, we have for all $(x,y)\in \R^{n_x}\times \R^{n_y}$:
    \begin{align*}
        \|\nabla_x f(x,y) \|^2 &\underset{\mathclap{(\text{PL})}}{\geq} 2 \mu_1\left[f(x, y)-\min _x f(x, y)\right] \\
        & \underset{\mathclap{(\text{Lem. }\ref{lem:PLgrowth})}}{\geq} \mu_1^2 d(x,\BR_x(y))^2.
    \end{align*}
    It follows that $f$ satisfies the EB property globally and thus by \Cref{prop:EB}, $E$ is connected.
\end{proof}

\subsection{Proof of \Cref{prop:incave_potential}}
\label{sec:proof:incave_potential}
\begin{proof}
    The existence follows as $P$ is continuous and decreasing at infinity and thus attains its global maximum. The connectedness follows from \Cref{thm:main} since the set of Nash equilibria is equivalent to the set of maxima of $P$.
\end{proof}
\subsection{Proof of \Cref{prop:connected_k_rationalizable}}
\label{sec:proof:connected_k_rationalizable}
\begin{proof}
    We show that for any connected set $S$, $\lambda(S)$ is connected. The proof for general $k$ then follows by induction. \Cref{lem:connectedBRgraph} shows that the graphs $\Gr(\BR_i|_{\pi_{-i}(S)})$ are connected. Since the projection of $\Gr(\BR_i|_{\pi_{-i}(S)})$ onto the image of $\BR_i|_{\pi_{-i}(S)}$ is a continuous surjective function, it follows that the sets $\cup_{a \in S}\BR_i(\pi_{-i}(a))$ are connected for all $i$ and therefore so is $\lambda(S)$, being the finite product of these sets.
\end{proof}

\subsection{Proof of \Cref{prop:connected_rationalizable}}
\label{sec:proof:connected_rationalizable}
\begin{proof}
    
   Since $\BR_i$ is outer semicontinuous (cf. \Cref{lem:osc}), it follows that its image $\BR_i(\pi_{-i}(K))$ is closed (cf. \citet{Dontchev2014Implicit}) and compact (since $K$ is compact and $\lambda(K) \subseteq K$ is bounded.) and therefore so is $\lambda(K)$ and by induction all $\lambda^k(S)$. From \Cref{prop:connected_k_rationalizable} we know that $\lambda^k(K)$ is connected for all $k\in \N$. Furthermore, we note that $\lambda$ respects set inclusion, i.e. $A \subset B \implies \lambda(A)\subset \lambda(B)$. Therefore, $\lambda^{k+1}(K)\subseteq \lambda^{k}(K)$ for all $k \in \N$. As $\lambda(K)\subseteq \R^{\sum_i^n n_i}$ is compact and thus normal, it follows from \Cref{lem:finiteintersection} that the intersection $R(K)=\bigcap_{k=1}^\infty \lambda^k(K)$ is non-empty, compact and connected.
\end{proof}

\subsection{Proof for Claim in  \Cref{fig:3d_plot_sine}}
\label{sec:proof:ex:2PLfunction}
For $x,y \in \R$, define $a = \max\left(|x| - 1, 0\right)$ and $b = \max\left(|y| - 1, 0\right)$ the function is given as:
    \[
        f(x,y)= a^2 +3 \sin^2(b)\sin^2(a) - 4b^2-10\sin^2(b),
    \]

We need to verify the assumptions of \Cref{prop:PLconnected}, i.e. $f$ is $\mathcal{C}^1$, two-sided PL, Lipschitz gradients, and bounded sets of minima and maxima.

\paragraph{Bounded Sets of Minima and Maxima:} It is easy to check that for any $x,y$, $\BR_x(y)=[-1,1]$ and $\BR_y(x)=[-1,1]$. Therefore, the sets of minima and maxima are bounded.

\paragraph{Lipschitz Gradients and $\mathcal{C}^1$:} The derivatives of $f$ are given by:
\begin{align*}
    \nabla_x f(x,y) = & \left(2a + 3\sin(2a)\sin^2(b)\right) \sgn(x)\\
    \nabla_y f(x,y) = & \left(-8b + 3 \sin^2(a)\cos(2b) - 10\sin(2b)\right) \sgn(y)
\end{align*}
As both partial derivatives exist and are continuous, $f$ is $\mathcal{C}^1$. To show Lipschitz continuity, we compute the second derivatives:
\begin{align*}
    |\nabla_{xx} f(x,y)| = & \left|2 + 6\sin(2a)\sin^2(b)\right| \\ 
    |\nabla_{yy} f(x,y)| = & \left|-8 + 6\sin^2(a)\cos(2b) - 20\cos(2b)\right|\\
    |\nabla_{xy} f(x,y)| = & \left|3 \sin(2a)\sin(2b)\right|\\
\end{align*}
All second derivatives are bounded by $28$ and exist with the exception of the null set $\{(x,y): |x|=1 \text{ or } |y|=1\}$. Therefore, $\nabla f$ is Lipschitz continuous.

\paragraph{Two-sided PL:} We need to show that $f$ satisfies the two-sided PL condition. 
We start by verifying the condition for \( f(\cdot,y) \) given any \( y \).
Let \( \mu_1 = \frac{1}{32} \). 
We need to show
\[
\|\nabla_x f(x,y)\|^2 \geq 2 \mu_1 \left(f(x,y) - \min_{x'} f(x',y)\right).
\]
First, note
\[
f(x,y) - \min_{x'} f(x',y) = a^2 + 3 \sin^2(a)\sin^2(b) \leq 4a^2.
\]
Furthermore,
\[
\|\nabla_x f(x,y)\|^2 = \left(2a + 3\sin(2a)\sin^2(b)\right)^2.
\]
We claim
\[
2a + 3\sin(2a)\sin^2(b) \geq a/2.
\]
Note that \( \sin^2(b) \in [0,1] \) and the lowest value will be achieved for \( \sin^2(b) = 1 \). Reparameterizing with \( t = 2a \), we get:
\[
t + 3\sin(t) \geq \frac{t}{4} \quad \Longleftrightarrow \quad \sin(t) \geq -\frac{t}{4},
\]
which can easily be verified.
We therefore conclude
\[
\|\nabla_x f(x,y)\|^2 \geq \frac{a^2}{4} \geq \frac{1}{32}\cdot 4a^2 \geq 2\mu_1\cdot \left(f(x,y) - \min_{x'} f(x',y)\right).
\]
Next, we verify the condition for \( f(x,\cdot) \) given any \( x \). We have that
\begin{align*}
\max_y f(x,y) - f(x,y) &= -3\sin^2(a)\sin^2(b) + 4b^2 + 10\sin^2(b)\\
&\leq 4b^2 + 10\sin^2(b)\\
&\leq 14b^2.
\end{align*}
At the same time we have
\begin{align*}
\|\nabla_y f(x,y)\|^2 = \left(-8b + 3\sin^2(a)\cos(2b) + 10\sin(2b)\right)^2 \geq 4b^2.
\end{align*}
We thus get for $\mu_1 = \frac{1}{7}$ that for all $x,y$:
\begin{align*}
\|\nabla_y f(x,y)\|^2 \geq 4b^2 \geq 2 \cdot \frac{2}{14} \cdot 14b^2 \geq 2\mu_1\left(\max_x f(x,y) - f(x,y)\right).
\end{align*}

\subsection{Proof for Claim in \Cref{ex:localPLconnected}}
\label{sec:proof:ex:localPLconnected}
 For $x,y \in \R$, define $a = \max\left(|x| - 1, 0\right)$ and $b = \max\left(|y| - 1, 0\right)$. We need to show the function 
    \[
        f(x,y)= a^2 \exp\left(-b^2\right) - b^2,
    \]
satisfies \Cref{assum:localPL}, is invex-incave, increasing at infinity in $x$, and decreasing at infinity in $y$. 

\paragraph{Invex-Incave and Increasing/Decreasing at Infinity:} 
Note that $\BR_x(y)=[-1,1]$ and $\BR_y(x)=[-1,1]$ for all $x,y\in \R$. As $f$ is strictly increasing in $|x|$ for $|x|>1$ and strictly decreasing in $|y|$ for $|y|>1$, it follows that it is invex-incave and increasing at infinity in $x$ and decreasing at infinity in $y$.

\paragraph{$\mathcal{C}^1$ and Lipschitz Gradients:} The derivatives of $f$ are given by:
\begin{align*}
    \nabla_x f(x,y) = & 2a \exp\left(-b^2\right) \sgn(x)\\
    \nabla_y f(x,y) = & \left(-2b a^2 \exp\left(-b^2\right) - 2b \right)\sgn(y)= -2b\left(a^2 \exp\left(-b^2\right) + 1\right)\sgn(y).
\end{align*}
As both partial derivatives exist and are continuous, $f$ is $\mathcal{C}^1$. To show Lipschitz continuity, we compute the second derivatives:
\begin{align*}
    |\nabla_{xx} f(x,y)| = & 2 \exp\left(-b^2\right) \\
    |\nabla_{yy} f(x,y)| = & \left|2a^2 \exp\left(-b^2\right) \left(2b^2 - 1\right) -2\right|\\
    |\nabla_{xy} f(x,y)| = & 4ab \exp\left(-b^2\right).
    \end{align*}

For \Cref{assum:localPL}, we only need to show Lipschitz continuity of the gradients in a neighborhood of $\underline{M}\cup \overline{M}\subseteq[-1,1]\times[-1,1]$. We choose the neighborhood $N=[-2,2]\times[-2,2]$ and note that on $N$, $a\in [0,1]$ and $b\in [0,1]$. The fact that both $a$ and $b$ are upper bounded by $1$ will be used often throughout the proof. As the second derivatives exist on $N$, except on a set with Lebesgue measure zero, and are all bounded by $4$, it follows that $\nabla f$ is Lipschitz continuous on $N$.

\paragraph{Local PL Condition:} We show that $f(\cdot,y)$ satisfies the two sided PL condition on $N$ uniformly with $\alpha_y=2$ and $\mu_y=\frac{4}{e}$. We have that:
\[
\|\nabla_x f(x,y)|^2 = 4a^2 \exp\left(-2b^2\right)
\]
and
\[
f(x,y) - \min_{x'} f(x',y) = a^2 \exp\left(-b^2\right) .
\]
It follows that:
\begin{align*}
    \|\nabla_x f(x,y)\|^2 = & 4a^2 \exp\left(-2b^2\right) \\
    \geq & 4a^2 \exp\left(-b^2\right) \frac{1}{e}\\
    = & 2\mu_y\left(f(x,y) - \min_{x'} f(x',y)\right).
\end{align*}
Next, we show that $f(x,\cdot)$ satisfies the PL condition uniformly on $N$ with $\alpha_x=2$ and $\mu_x=1$. We have that:
\begin{align*}
    \|\nabla_y f(x,y)\|^2 = & 4b^2\left(a^2 \exp\left(-b^2\right) + 1\right)^2\\
    \end{align*}
and
\[
\max_y f(x,y) - f(x,y) = a^2 (1-\exp\left(-b^2\right)) + b^2.
\]
Using $b^2 \geq 1- \exp\left(-b^2\right)$, it follows that:
\begin{align*}
    \|\nabla_y f(x,y)\|^2 = & 4b^2\left(a^2 \exp\left(-b^2\right) + 1 \right)^2\\
    \geq & 4b^2 + 4b^2a^2 \frac{1}{e}\\
    \geq & a^2b^2 + b^2
    \geq a^2(1-\exp\left(-b^2\right)) + b^2\\
    = & 2\mu_x\left(\max_y f(x,y) - f(x,y)\right).
\end{align*}
\paragraph{Quadratic Growth:} Finally, we show that $f(\cdot,y)$ and $f(x,\cdot)$ satisfy the quadratic growth condition on $N$. Let us start with $f(\cdot,y)$. We have that:
\[
 f(x,y) - \min_{x'} f(x',y)  = a^2 \exp\left(-b^2\right) 
\]
and 
\[
d(x,\BR_x(y)) = a.
\]
Setting $\eta_y = \frac{1}{e}$, we conclude that:
\begin{align*}
    f(x,y) - \min_{x'} f(x',y) = & a^2 \exp\left(-b^2\right) \\
    \geq & \frac{1}{e} a^2 \\
    = & \eta_y d(x,\BR_x(y))^2.
\end{align*}
Similarly, setting $\eta_x = 1$, we have that:
\begin{align*}
    \max_{y'} f(x,y') - f(x,y) = & a^2 (1-\exp\left(-b^2\right)) + b^2\\
    \geq & b^2\\
    = & \eta_x d(y,\BR_y(x))^2.
\end{align*}
and thus $f$ satisfies all assumptions made.

\subsection{Auxiliary Results}
\label{sec:auxiliaryresults}

\begin{lemma}[Interchangeability Lemma,\citep{Nash1951Non}]
    \label{lem:addeq}
        Let $(x,y),(x',y')\in E$, then mixing the strategies yields two additional equilibria, i.e. $(x,y'),(x',y)\in E$.
    \end{lemma}
    \begin{figure}[H]
        \centering
        \begin{tikzpicture}
            \coordinate (A) at (0,2);
            \coordinate (B) at (2,2);
            \coordinate (C) at (0,0);
            \coordinate (D) at (2,0);

            \node[fill, circle, inner sep=1pt, label=left:{$(x, y')$}] (A) at (A) {};
            \node[fill, circle, inner sep=1pt, label=right:{$(x', y')$}] (B) at (B) {};
            \node[fill, circle, inner sep=1pt, label=left:{$(x, y)$}] (C) at (C) {};
            \node[fill, circle, inner sep=1pt, label=right:{$(x', y)$}] (D) at (D) {};
    
            \draw (A) -- (B);
            \draw (A) -- (C);
            \draw (B) -- (D);
            \draw (C) -- (D);
    
        \end{tikzpicture}
        \caption{Mixing coordinates creates new equilibria.}
        \label{fig:equilibrium1d}
    \end{figure}
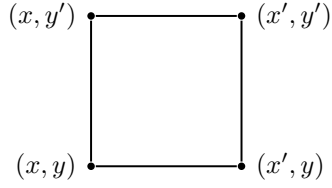
\begin{proof}[Proof]

        See \Cref{fig:equilibrium1d} for the intuition of the proof.
            Repeatedly applying the equilibrium definition, we get the following:
            \[
            f(x,y)\underset{x \in \BR_x(y)}{\leq} f(x',y) \underset{y' \in \BR_y(x')}{\leq} f(x',y').
            \]
            At the same time it holds that:
                \[
            f(x',y')\underset{x' \in \BR_x(y')}{\leq} f(x,y') \underset{y \in \BR_y(x)}{\leq} f(x,y).
            \]
            We thus conclude
            \[
            f(x,y)=f(x',y)=f(x,y')=f(x',y').
            \]
        In particular, since $(x,y),(x',y')\in E$ and $f(x,y)=f(x',y)=f(x',y') $, $x'$ must be a best response to $y$ and vice versa. Therefore $(x',y) \in E$. The argument holds equivalently for $(x,y')$.

\end{proof}

\begin{lemma}[Product Structure of Saddle Points, \citep{Bertsekas2014Convex,Rockafellar1970Convex}]
    \label{prop:eqproduct}
    If E is non-empty, then $E=X\times Y$ is the product of the primal and dual optimal set.
\end{lemma}

\begin{proof}
    \Cref{lem:addeq} in \Cref{app:proofs} proves that $E=A\times B$ for some sets $A\subset \R^{n_x}, B\subset \R^{n_y}$. As every Nash equilibrium is also a minimax, respectively maximin point, it follows that $E= A\times B \subset X \times Y$.
    
    To show the converse, take any $(x^*,y^*) \in X \times Y$. As $E$ is non-empty, it holds that $\min_x \max_y f(x,y)=\max_y \min_x f(x,y)=f^*$ and therefore
    \[
    f^*=\min_x f(x,y^*)\leq f(x^*,y^*)\leq \max_y f(x^*,y)=f^*.
    \]
    We conclude $(x^*,y^*) \in E$ and thus $E= X\times Y$.
\end{proof}

\begin{definition}[Local Error Bound]
$f$ satisfies the local error bound (EB) at a point $(\overline{x},\overline{y})$, if there exists constants $\nu_{\overline{x}}>0$ and $\nu_{\overline{y}}>0$, as well as neighborhoods $N_{\overline{x}}$ of $\overline{x}$ and $N_{\overline{y}}$ of $\overline{y}$, such that:
    \begin{align*}
        \forall x \in N_{\overline{x}} \forall y \in N_{\overline{y}}: &\nu_{\overline{x}} \|\nabla_x f(x,\overline{y})\| \geq d(x,\BR(\overline{y}))\\
        &\nu_{\overline{y}}\|\nabla_y f(\overline{x},y)\| \geq  d(y,\BR(\overline{x})).
    \end{align*}
\end{definition}

\begin{lemma}
    \label{prop:EB}
    Assume $f$ satisfies the local $EB$ property for all $(\overline{x},\overline{y})\in \underline{M} \cup \overline{M}$. Further, assume that $\nabla_y f(\cdot,\overline{y})$ is locally $L_{\overline{y}}$-Lipschitz continuous in $x$ at all $(\overline{x},\overline{y})\in \underline{M}$ and that $\nabla_x f(\overline{x},\cdot)$ is locally $L_{\overline{x}}$-Lipschitz continuous in $y$ at all $(\overline{x},\overline{y})\in \overline{M}$. Then $E=\underline{M}=\overline{M}$ and this common set is connected.
\end{lemma}

\begin{proof}
    We show that $f$ locally satisfying the local EB property and $\nabla_y f(\cdot,\overline{y})$ being locally $L_{\overline{y}}$-Lipschitz continuous imply that $\BR_y(x)$ is \lilc\ and thus $\underline{M}\subseteq E$. The proof for $\overline{M}$ follows analogously, which implies that $E=\underline{M}=\overline{M}$ and thus the connectedness of this common set. The local Lipschitz continuity of $\nabla_y f(\cdot,\overline{y})$ at $(\overline{x},\overline{y})$ holds for a constant $L_{\overline{y}}$ and a neighborhood $N_{\overline{x}}^{\text{L}}$ of $\overline{x}$, while the local EB property holds for a constant $\nu_{\overline{y}}$ and a neighborhood $N_{\overline{x}}^{\text{EB}}\times N_{\overline{y}}^{\text{EB}}$ of $(\overline{x},\overline{y})$. Define $N_{\overline{x}}:=N_{\overline{x}}^{\text{EB}}\cap N_{\overline{x}}^{\text{L}}$ and $\kappa:= \nu_{\overline{y}} L_{\overline{y}}$. Then, it holds for all $x\in N_{\overline{x}}$:
    \begin{align*}
        d(\overline{y},\BR_y(x)) &\underset{\mathclap{(\text{EB})}}{\leq} \nu_{\overline{y}} \|\nabla_y f(x,\overline{y})\| \\
        &= \nu_{\overline{y}} \|\nabla_y f(x,\overline{y})-\nabla_y f(\overline{x},\overline{y})\| \\
        & \underset{\mathclap{\text{(Lip.)}}}{\leq} \nu_{\overline{y}} L_{\overline{y}} \|x-\overline{x}\| \\
        & = \kappa \|x-\overline{x}\|,
    \end{align*}
    where the first equality holds because $\overline{y}\in \BR_y(\overline{x})$ and $f(\overline{x},\cdot)$ is incave, which implies $\nabla_y f(\overline{x},\overline{y})=0$.
   We conclude that $\BR_y(x)$ is \lilc\ at $\overline{x}$ for $\overline{y}$ for all $(\overline{x},\overline{y})\in \underline{M}$ and thus $\underline{M}\subseteq E$ and $E$ is connected.
\end{proof}

The Lemma below is related to the well-known Heine-Cantor theorem~\citep{Heine1872Die}.
\begin{lemma}[Uniform Continuity on Compact Set]
    \label{lem:uniformcont}
    Let $f:\R^{n_x}\times \R^{n_y} \to \R, (x,y)\mapsto f(x,y)$ be a continuous function. Let $K\subset \R^{n_y}$ be compact. Then $f(\cdot,y)$ is uniformly continuous on $K$ with respect to $y$ , i.e. for all $\overline{x} \in \R^{n_x}$ and $\epsilon>0$, there exists a $\delta>0$, such that for all $y\in K$ and for all $x\in \mathbb{B}_\delta(\overline{x})$:
    \[
    |f(x,y)-f(\overline{x},y)|<\epsilon,
    \]
    where $\mathbb{B}_\delta(\overline{x})\subset \R^{n_x}$ denotes the open ball of radius $\delta$ around $\overline{x}$.
\end{lemma}
\begin{proof}
    Let $\overline{x}\in \R^{n_x}$ and $\epsilon>0$. Consider any point $y\in K$. By the continuity\footnote{Here we use the product of open balls in $\R^{n_x}$ and $\R^{n_y}$ to define continuity instead of directly using open balls in $\R^{n_x}\times \R^{n_y}$. The two are of course equivalent, since the products of open sets form the basis of the product topology~\citep{Hatcher2005Notes}.} of $f$, there exist a $\gamma_y,\chi_y>0$, such that for all $x'\in \mathbb{B}_{\chi_y}(\overline{x})$ and all $y'\in \mathbb{B}_{\gamma_y}(y)$:
    \[
        |f(x',y')-f(\overline{x},y)|<\epsilon/2.
    \]
    For the given $\epsilon$, let us denote for each $y\in K$, these open neighborhoods by $N_y= \{(x',y'): x'\in \mathbb{B}_{\chi_y}(\overline{x}), y'\in \mathbb{B}_{\gamma_y}(y)\}$. The collection $\{N_y: y\in K\}$ forms an open cover of $\overline{x} \times K$ and thus, by compactness. there exists a finite subcover $\{N_{y_1},\dots,N_{y_m}\}$. We define $\delta = \min_{1\leq i \leq m} \chi_{y_i}$ and claim that $\delta$ satisfies the condition of the lemma.

    Indeed, let $y\in K$ and $x\in \mathbb{B}_\delta(\overline{x})$. By construction there exists a $y_i$ such that $y \in \B_{\gamma_{y_i}}(y_i)$ and thus $(x,y)\in N_{y_i}$. As $\delta\leq \chi_{y_i}$, it follows that:
    \[
    | f(x,y)-f(\overline{x},y_i) | \leq \epsilon/2.
    \]
    Applying the triangle inequality yields the desired result:
    \begin{align*}
        | f(x,y)-f(\overline{x},y) | \leq & | f(x,y)-f(\overline{x},y_i) | + | f(\overline{x},y_i)-f(\overline{x},y) |\\
        \leq & \epsilon/2 + \epsilon/2 \\
        =& \epsilon.
    \end{align*}
\end{proof}

For the next result, we need to introduce the notion of outer semicontinuity.
\begin{definition}[Outer Semicontinuity]
    \label{def:osc}
    Let $S: \R^n \rightrightarrows \R^m$ be a set-valued map. We say $S$ is outer semicontinuous at $\overline{x}\in \R^n$, if for all open sets $V$, such that $S(\overline{x})\subset V$, there exists an open set $U$ such that $\overline{x}\in U$ and $S(x)\subset V$ for all $x\in U$. We say $S$ is outer semicontinuous if it is outer semicontinuous at all $\overline{x}\in \R^n$.
\end{definition}

\begin{lemma}[Outer Semicontinuity of Best Response Maps]
    \label{lem:osc}
    Consider an incave game, as described in \Cref{sec:incavegames}. The best response maps $\BR_i(\cdot)$ are outer semicontinuous.
\end{lemma}
\begin{proof}
    Let $\bar{a}_{-i} \in \prod_{j\neq i} A_j$ and let $V \supset \BR_i(\bar{a}_{-i})$ be an open set in $\R^{n_i}$. By compactness of $\BR_i(\bar{a}_{-i})$, there exists an $\epsilon>0$, such that:
    \[
    N_\epsilon= \B_{\epsilon}(\BR_i(\bar{a}_{-i})) = \{ a_i \in A_i: d(a_i,\BR_i(\bar{a}_{-i}))<\epsilon \} \subset V.
    \]

    Let $\bar{a}_i \in \BR_i(\bar{a}_{-i})$. It follows from compactness of the boundary $\partial N_\epsilon$ that there exists a $c_1$, such that
    \[
    \forall a_i\in \partial N_\epsilon: u_i(\bar{a}_i,\bar{a}_{-i}) > u_i(a_i,\bar{a}_{-i}) +c_1.
    \]
    Moreover, \Cref{lem:uniformcont} applied to the compact set $\partial N_\epsilon$ implies that for all $\epsilon$, there exists a $\delta>0$, such that:
    \[
    \forall a_{-i}\in \B_{\delta}(\bar{a}_{-i}), \forall a_i\in \partial N_\epsilon: |u_i(a_i,a_{-i}) - u_i(a_i,\bar{a}_{-i})| < \epsilon.
    \]

    Combining the two inequalities, we conclude that there exist  $c_2,\delta_1$, such that:
    \[
    \forall a_{-i}\in \B_{\delta_1}(\bar{a}_{-i}), \forall a_i\in \partial N_\epsilon: u_i(\bar{a}_i,\bar{a}_{-i}) > u_i(a_i,a_{-i}) + c_2.
    \]
    At the same time from the continuity of $u_i$, it follows that there exists a $\delta_2>0$, such that:
    \[
    \forall a_{-i}\in \B_{\delta_2}(\bar{a}_{-i}): u_i(\bar{a}_i,a_{-i}) > u_i(\bar{a}_i,\bar{a}_{-i}) - c_2/2.
    \]
    Let $\delta=\min(\delta_1,\delta_2)$. It follows that for all $a_{-i}\in \B_{\delta}(\bar{a}_{-i})$, the utility achieved by playing $\bar{a}_i$ is greater than the utility achieved by playing any $a_i\in \partial N_\epsilon$. Combining this with the connectedness of superlevel sets of $u_i(\cdot,\bar{a}_{-i})$ (cf. \Cref{thm:main}), we conclude that for all $a_{-i}\in \B_{\delta}(\bar{a}_{-i})$, it holds that $\BR_i(a_{-i})\subset N_\epsilon \subset V$ and thus $\BR_i(a_{-i})$ is outer semicontinuous.
\end{proof}

\begin{definition}[Inner semicontinuity]
    \label{def:isc}
    Let $S:\R^{n} \rightrightarrows \R^{m}$ be a set-valued map. We say $S$ is inner semicontinuous at $\overline{x} \in \R^n$, if for all open sets $V$, such that $S(\overline{x})\cap V \neq \emptyset$, there exists an open set $U$, such that $\overline{x}\in U$ and for all $x\in U$, it holds that $S(x)\cap V \neq \emptyset$.
\end{definition}

\begin{lemma}
    \label{lem:isc}
    Under \Cref{assum:localPL}, the set-valued maps $\BR_y(x)$ and $\BR_x(y)$ are inner semicontinuous at all $\overline{x}\in X$, respectively all $\overline{y}\in Y$.
\end{lemma}
\begin{proof}
WLOG we restrict to showing that $\BR_y(x)$ is inner semicontinuous at a given $\overline{x}\in X$. The proof for $\BR_x(y)$ follows analogously.
We can rephrase the inner semicontinuity condition using open balls as follows: 
\[
\forall \epsilon>0\; \forall \overline{y}\in \BR_y(\overline{x})\; \exists \delta>0\; \forall x \in \B_{\delta}(\overline{x}): \BR_y(x) \cap \B_{\epsilon}(\overline{y}) \neq \emptyset.
\]

For contradiction, suppose that $\exists \overline{y}\in \BR_y(\overline{x})$ and $\epsilon_1$, such that $\forall \delta$, there exists an $x_\delta \in \B_{\delta}(\overline{x})$ with $\BR_y(x_\delta) \cap \B_{\epsilon_1}(\overline{y}) = \emptyset$. In other words, for all $\delta>0$, there exists an $x_\delta \in \B_{\delta}(\overline{x})$ such that $d(\overline{y}, \BR_y(x_\delta)) \geq \epsilon_1$.

From \Cref{assum:localPL} we can use the local growth condition to deduce that for $\delta$ small enough such that $\B_{\delta}(\overline{x})\subset N_{\overline{x}}^{PLG}$, the following holds:
\[
\max_y f(x_\delta,y)-f(x_\delta,\overline{y}) \geq \eta_x d(\overline{y},\BR_y(x_\delta))^{\frac{\alpha_x}{\alpha_x-1}} \geq \eta_x \epsilon_1^{\frac{\alpha_x}{\alpha_x-1}}=k.
\] For brevity, we define $k:=\eta_x \epsilon_1^{\frac{\alpha_x}{\alpha_x-1}}$. Note that $k$ is independent of our choice of $\delta$.

Let us choose any $\epsilon_2>0$. From \Cref{lem:osc}, we conclude that $\BR_y(x)$ is outer semicontinuous at $\overline{x}$.\footnote{While \Cref{lem:osc} was written for the more general setting of incave games, it captures invex-incave minimax optimization as a special case.} In particular, there exists a $\delta_2>0$, such that for all $x\in \B_{\delta_2}(\overline{x})$:
\[
\BR_y(x) \subset \B_{\epsilon_2}(\BR_y(\overline{x})).
\]
Consider the compact set $K= \overline{\B_{\epsilon_2}(\BR_y(\overline{x}))}$. We invoke \Cref{lem:uniformcont}, which states:
\begin{equation}
    \label{eq:uniformconteps}
    \forall \epsilon \; \exists \delta \; \forall y \in K \; \forall x \in \B_{\delta}(\overline{x}): |f(x,y)-f(\overline{x},y)|<\epsilon.
\end{equation}

We choose a $\delta\leq \delta_2$ and let $\epsilon_\delta$ be the corresponding $\epsilon$ for which \eqref{eq:uniformconteps} holds. Further, let us denote by $y_\delta$ an arbitrary point in $\BR_y(x_{\delta})$. By definition of $\overline{y}\in \BR_y(\overline{x})$, we know that $f(\overline{x},\overline{y})\geq f(\overline{x},y_\delta)$. However, at the same time, we have that $\forall \delta \leq \delta_2$:
\begin{align*}
    f(\overline{x},y_{\delta})+ \epsilon_\delta &\geq f(x_\delta,y_\delta)\\
    &\geq f(x_\delta,\overline{y}) + k\\
    &\geq f(\overline{x},\overline{y}) + k - \epsilon_\delta.
\end{align*}
The first and third inequalities follow from the uniform continuity of $f$ and its outer semicontinuity, whereas the second inequality follows from the growth condition discussed beforehand. Rearranging the inequalities yields:
\[
\forall \delta \leq \delta_2: f(\overline{x},y_\delta) + 2 \epsilon_\delta \geq f(\overline{x},\overline{y}) + k.
\]
However, $\epsilon_\delta \to 0$ as $\delta \to 0$, and thus for sufficiently small $\delta$, $f(\overline{x},y_\delta) > f(\overline{x},\overline{y}) $, 
 such that we reach a contradiction with $f(\overline{x},\overline{y})\geq f(\overline{x},y_\delta)$. We therefore conclude that $\BR_y(x)$ is inner semicontinuous at $\overline{x}$.
\end{proof}

\begin{lemma}[Connectedness of Best Response Graph]
    \label{lem:connectedBRgraph}
    In an incave game, as described in \Cref{sec:incavegames} and given a connected set $S\subseteq \mathcal{A}$, the best response graph $\Gr(\BR_i|_{\pi_{-i}(S)})$ is connected. 
\end{lemma}

\begin{proof}
    First note, that since $\pi_{-i}$ is a continuous functions and $S$ is connected, it follows that $\pi_{-i}(S)$ is connected.
    Next, for contradiction, suppose there exists two disjoint open sets $A,B$ such that $\Gr(\BR_i|_{\pi_{-i}(S)})=A\cup B$. 

    Let us define the sets:
    \[
        \begin{aligned}
            A_S &= \{a_{-i} \in \pi_{-i}(S): (a_i,a_{-i})\in A\}\\
            B_S &= \{a_{-i} \in \pi_{-i}(S): (a_i,a_{-i})\in B\}.
        \end{aligned}
    \]
    By construction, $A_S$ and $B_S$ are both non-empty and $\pi_{-i}(S)= A_S \cup B_S$. Our aim is to show that the sets are open and disjoint to reach a contradiction to the connectedness of $\pi_{-i}(S)$.

    Let us start with showing they are disjoint. Indeed, for contradiction, suppose that there exists $a_{-i}'$ such that $a_{-i}' \in A_S \cap B_S$. In this case, there exist disjoint, open sets $A\cap \left(\BR_i(a_{-i}')\times\{a_{-i}'\}\right)$ and $B\cap \left(\BR_i(a_{-i}')\times\{a_{-i}'\}\right)$ whose union is the connected set $\BR_i(a_{-i}')\times\{a_{-i}'\}$ (again using \Cref{thm:main} applied to the incavity of $u_i$), a contradiction. It follows that $A_S$ and $B_S$ are disjoint.
    
    Next, we show that $A_S$ is open (the proof for $B_S$ follows analogously). Let ${a}_{-i}' \in A_S$. By construction, this implies the compact set $\BR_i(a_{-i}')\times\{a_{-i}'\}$ is a subset of the open set $A$. Applying the generalized tube lemma~\citep{Munkres2014Topology}, there exist $\epsilon_1,\epsilon_2>0$, such that:
    \[
        \{ (a_i,a_{-i}): a_i\in \B_{\epsilon_1}(\BR_i(a_{-i}')), a_{-i}\in \B_{\epsilon_2}(a_{-i}') \} \subset A.
    \]
    As shown in \Cref{lem:osc}, the best response map $\BR_i(\cdot)$ is outer semicontinuous. Therefore, there exists a $\delta$ with $0<\delta<\epsilon_2$, such that for all $a_{-i}\in \B_\delta(a_{-i}')$:
    \begin{align*}
        &\BR_i(a_{-i}) \subseteq  \B_{\epsilon_1}(\BR_i(a_{-i}')) \\
        \iff & \BR_i(a_{-i})\times\{a_{-i}\} \subseteq  A\\
        \iff & a_{-i} \in  A_S.
    \end{align*}
    We conclude that $A_S$ is open. The same argument holds for $B_S$, which would imply $\pi_{-i}(S)$ is disconnected, a contradiction. Therefore, $\Gr(\BR_i|_{\pi_{-i}(S)})$ is connected.
    
\end{proof}

\begin{lemma} [Connected Finite Intersection Property]
    \label{lem:finiteintersection}
    Let $\{ K_i \}_{i\in \N}$ be a collection of nested\footnote{That is $K_i\supset K_{i+1}$ for all $i$.}, compact, nonempty, connected sets inside a normal subspace $X\subseteq\R^n$. We say that the sets have the \emph{finite intersection property}, i.e. $\forall n \in \N: \bigcap_{i=1}^n K_i \neq \emptyset$.
    Denote by $K_\infty = \bigcap_{i=1}^\infty K_i$ the intersection of all $K_i$. Then:
    \begin{enumerate}
        \item $K_\infty$ is non-empty and compact.
        \item $K_\infty$ is connected.
    \end{enumerate}
\end{lemma}
\begin{proof}
    The first property is a standard result in topology, see e.g. \citet[Thm. 26.9]{Munkres2014Topology}. For the second property, assume for contradiction that $K_\infty$ is disconnected. Then there exist two nonempty, disjoint, closed sets $A,B$, such that $K_\infty = A \cup B$. Since $K_\infty$ is closed, it follows that $A$ and $B$ are closed in $X$. By the normality of $X$, there exist open sets $O_A,O_B$ such that $A\subset O_A$, $B\subset O_B$ and $O_A \cap O_B = \emptyset$.
    Denote by $O = O_A \cup O_B$ the union of the two open sets. We claim there exists some $N$, such that for all $n\geq N$, $K_n \subset O$. 
    For contradiction, suppose this is not the case. Then, we can construct a sequence of non-empty, closed, nested sets $\{C_m= K_m \setminus O\}$. It follows from the first property that $C_\infty=\bigcap_{m=1}^\infty C_m \neq \emptyset$ and by construction $C_\infty \subset K_\infty\setminus O = \emptyset$, a contradiction. Therefore, there exists some $N$, such that for all $n\geq N$, $K_n \subset O$.
    But in this case, $O_A,O_B$ disconnect $K_n$, which is a contradiction to the assumption that $K_n$ is connected. Therefore, $K_\infty$ is connected.
\end{proof}

For the next Lemma, let $P$ denote the set of minima of a function $f$. We say that a continuously differentiable function $f:\R^n \to \R$ locally satisfies the $\alpha$-PL condition (cf. \Cref{def:aPL}) around the set of minima $P$, if there exists an open neighborhood $N$ of $P$, such that $P \subset N$ and a constant $\mu>0$, such that:
    \[
    \forall x \in N: \|\nabla f(x)\|^\alpha \geq \mu \left(f(x) - \min f(x)\right).
    \]

\begin{lemma}[$\alpha$-PL implies $\frac{\alpha}{\alpha-1}$-growth] 
    \label{lem:PLgrowth}  
    Let $f:\R^n \to \R$ be a differentiable function. If $f$ is invex and locally satisfies the $\alpha$-PL condition for a neighborhood $N^{\text{PL}}$ and $\alpha>1$, then there exists a neighborhood $N^{\text{G}}\subset N^{\text{PL}}$ such that:
    \begin{equation}
    \label{eq:growth}
    \forall x \in N^{\text{G}}: f(x)-\min_{x'}f(x')\geq \left(\frac{\alpha-1}{\alpha}\right) ^{\frac{\alpha}{\alpha-1}} \mu^{\frac{1}{\alpha-1}} d(x,P)^\frac{\alpha}{\alpha-1},
    \end{equation}
    where $P=\{x:x\in \arg\min f(x)\}$ the set of minimizers of $f$.
    Therefore, $f$ locally has $\frac{\alpha}{\alpha-1}$-growth on the neighborhood $N^{\text{G}}$. If $f$ satisfies the $\alpha$-PL condition globally, then it also satisfies $\frac{\alpha}{\alpha-1}$-growth globally.
\end{lemma}
  
\begin{proof}
    This proof generalizes the one from \citet{Karimi2016Lineara} for global 2-PL functions.

    For points in $P$, \Cref{eq:growth} holds by definition, so let us consider $x \notin P$.

    Define, $g(x)=\left( f(x)- \min_{x'}f(x')\right)^{\frac{\alpha-1} {\alpha}}$. As $\alpha>1$, it follows by the local $\alpha$-PL condition that for all $x\in N^{\text{PL}}\setminus P$:

    \begin{equation}
        \label{eq:GPL}
            \begin{aligned}
        \norm{\nabla g(x)}\geq & \frac{\alpha-1}{\alpha} \left( f(x)- \min_{x'}f(x')\right)^{\frac{-1}{\alpha}} \|\nabla f(x)\|\\
        \geq & \frac{\alpha-1}{\alpha} \mu ^ {\frac{1}{\alpha}} 
    \end{aligned}
\end{equation}
    For $x \notin P$, consider the following ODE:
    \begin{align*}
        \frac{dx(t)}{dt}=&-\nabla g(x(t))\\
        x(0)=& x
    \end{align*}

    Note that $g(x(t))$ and therefore $f(x(t))$ is monotonically decreasing as $t$ increases. Recall, that $L_c^-=\{x': f(x')\leq c\}$. It follows that if $x(0) \in L_c^-$ then $x(t)\in L_c^-$ for all $t\geq 0$. We now construct an open neighbourhood $N^G$, such that if $x(0)\in N^G$, then $x(t)\in N^{\text{PL}}$ for all $t\geq 0$. 
    As the boundary $\partial N^{\text{PL}}$ is compact and $f$ is continuous, $f$ attains its minimum on $\partial N^{\text{PL}}$. Let $c$ be this minimum.
    By the continuity of $f$, for each $x_i\in P$, there exists a $\delta_{x_i}$, such that $\forall x\in \B_{\delta_{x_i}}(x_i): f(x) < c'=\frac{\min_x f(x)+c}{2}$. Since $P$ is compact, there exists a finite subcover $\{B_{\delta_{x_1}}(x_1),\dots,B_{\delta_{x_m}}(x_m)\}$ of $P$, whose union we denote by $N^G$. By construction $N^G$ is open and $P \subset N^G \subset L_{c'}^- \subset N^{PL}$. We deduce that for all $x \in N^{\text{G}}$, the trajectory $x(t)$ stays in $N^{\text{PL}}$, where the $\alpha$-PL condition holds. If the $\alpha$-PL condition is satisfied globally, then $N^{\text{G}}$ can be chosen to be the whole space $\R^n$.

    Next, we show that the ODE converges to $P$ in finite time.
    Indeed for any $\tau>0$, such that $x(\tau)\notin P$, we have for all starting points $x(0)=x\in N^{\text{G}}$ that: 
    \begin{align*}
        g(x(0))-g(x(\tau))=& - \int_{x(0)}^{x(\tau)} \langle \nabla g(x), dx \rangle \\
        = & - \int_{0}^{\tau} \langle \nabla g(x(t)), \frac{dx(t)}{dt} \rangle dt \\
        = & \int_0^\tau \norm{\nabla g(x(t))}^2 dt \\
        \geq & \int_0^\tau \left(\frac{\alpha-1}{\alpha}\right)^2 \mu ^ {\frac{2}{\alpha}} dt \\
        =& \tau \left(\frac{\alpha-1}{\alpha}\right)^2 \mu ^ {\frac{2}{\alpha}},
    \end{align*}
    where we use the $\alpha$-PL condition for the inequality.

    Since $g(x(\tau))\geq 0$, it follows that $\tau \leq g(x(0))\left(\frac{\alpha-1}{\alpha}\right)^{-2} \mu ^ {-\frac{2}{\alpha}} $ and therefore there has to exist a finite $T$ such that $x(T)\in P$.

    Note, the length of the path from $x$ to $x(T)$ is bounded from below by $d(x,P)$. Therefore, we have for all $x(0)=x\in N^{\text{G}}$:

    \begin{align*}
        g(x)-g(x(T))=& \int_0^T \norm{\nabla g(x(t))}^2 dt \\
        \geq & \frac{\alpha-1}{\alpha} \mu ^ {\frac{1}{\alpha}} d(x,P).
    \end{align*}
    As $g(x(T))=0$, it follows that for all $x\in N^{\text{G}}$:
    \[
    f(x)-f(x_p)=g(x)^{\frac{\alpha}{\alpha-1}}\geq \left(\frac{\alpha-1}{\alpha}\right)^{\frac{\alpha}{\alpha-1}} \mu^{\frac{1}{\alpha-1}} d(x,P)^\frac{\alpha}{\alpha-1}
    \]

\end{proof}

\section{Connectedness under Local Inner Hölder Continuity of Best Response Maps} 
\label{app:Holder}
Below, we give a second set of conditions on $\BR_x(y),\BR_y(x)$ and $f$---complementary to \Cref{thm:lilc}---
which guarantee $E=\underline{M}=\overline{M}$ and thus the connectedness of this common set. In \Cref{prop:EB,prop:PLconnected}, we assumed in addition to $f$ being continuously differentiable that the partial derivatives $\nabla_y f(\cdot,y)$ and $\nabla_x f(x,\cdot)$ are Lipschitz continuous in one variable. If we strengthen this assumption slightly to the total derivative $\nabla f$ being locally Lipschitz continuous in $x$ and $y$ on $\underline{M} \cup \overline{M}$, we can relax the assumption on the best response maps from local inner Lipschitz continuity to local inner Hölder continuity. 

\begin{definition}[Local Inner Hölder Continuity]
    \label{def:localholder}
    A set-valued map $S:\R^{n_x} \rightrightarrows \R^{n_y}$ is locally inner $\alpha$-Hölder continuous at $\overline{x}$ for $\overline{y}$, if $\overline{y}\in S(\overline{x})$ and there exist a constant $\kappa$ and an open neighborhood $N_{\overline{x}}$ of $\overline{x}$ such that:
    \[
    \forall x \in N_{\overline{x}}: d(\overline{y},S(x)) \leq \kappa \|x-\overline{x}\|^\alpha,
    \]
    
\end{definition}

In \Cref{thm:lilc}, when $f$ was only assumed to be continuously differentiable, the best response maps $\BR_y(x)$ and $\BR_x(y)$ had to be \lilc\ to control the error term in the first order approximation of $f$. If $f$ has Lipschitz gradients around $\underline{M}\cup\overline{M}$, we can leverage the fundamental theorem of calculus to get a more precise bound on the error term, such that even if the best response maps are only locally inner $\alpha$-Hölder continuous, we can still derive an envelope theorem for $F$ and $G$. It shows that they are differentiable and that for $x\in X:\nabla F(x)=\nabla_x f(x,y^*)$ for any $y^*\in \BR_y(x)$ and vice versa for $G$. It thus follows that for any $(x,y)\in \underline{M}$, it holds that $\nabla_x f(x,y)=\nabla F(x)=0$ and $(x,y)\in E$. Repeating the same argument for $\overline{M}$, we arrive at $E=\underline{M}=\overline{M}$ and the connectedness of this common set.

\begin{theorem}
\label{thm:Holderconnected}
Assume $f$ is continuously differentiable and that $\nabla f$ is locally Lipschitz continuous on $\underline{M}\cup\overline{M}$. Furthermore, assume that for all $(\overline{x},\overline{y})\in \underline{M}$, $\BR_y(x)$ is locally inner $\alpha_x$-Hölder continuous at $\overline{x}$ for $\overline{y}$ with $\alpha_x>0.5$ and conversely for all $(\overline{x},\overline{y})\in \overline{M}$, $\BR_x(y)$ is locally inner $\alpha_y$-Hölder continuous at $\overline{y}$ for $\overline{x}$ with $\alpha_y>0.5$. Then $E=\underline{M}=\overline{M}$ and this common set is connected.
\end{theorem}

\begin{proof}
    Analogous to the proof of \Cref{thm:lilc}, we restrict to showing that the local inner Hölder continuity of $\BR_y(x)$ implies that $\underline{M} \subseteq E$. The proof for $\overline{M}$ and thus the connectedness of $E$ follows analogously.

    Let $\overline{x}\in X$ and choose any $\overline{y}\in \BR_y(\overline{x})$. By the local inner Hölder continuity of $\BR_y(x)$, there exists a constant $\kappa_x$ and an open neighborhood $N_{\overline{x}}$ on which $\BR_y(x)$ is inner Hölder continuous. Let $\delta'$ be small enough, such that for all $\delta\leq \delta': \mathbb{B}_\delta (\overline{x})\subseteq N_{\overline{x}}$, where $\mathbb{B}_\delta (\overline{x})$ is an open ball of radius $\delta$ around $\overline{x}$, and let $v \in \R^{n_x}$ be a vector such that $\|v\|=1$. We choose a $y(\delta,v) \in \BR_y(\overline{x}+\delta v)$, such that $\|y(\delta,v)-\overline{y}\|=d(\overline{y},\BR_y(\overline{x}+\delta v))$. The latter is possible as $\BR_y(\overline{x}+\delta v)$ is compact.

    By the local inner Hölder continuity of $\BR_y({x})$ at $\overline{x}$ for $\overline{y}$, it holds that:
    \begin{align*}
       \|y(\delta,v)-\overline{y}\| =d(\overline{y},\BR_y(\overline{x}+\delta v)) \leq \kappa_x \|\delta v\|^{\alpha_x} = \kappa_x \delta^{\alpha_x}.
    \end{align*}

    Let $N^L$ be the neighborhood of $(\overline{x},\overline{y})$ on which $\nabla f$ is Lipschitz continuous, then if necessary we can choose a smaller $\delta'$, to ensure that for all $\delta\leq \delta': \mathbb{B}_\delta (\overline{x})\times \mathbb{B}_\delta (\overline{y})\subseteq N^L$.

    Using that $f$ is $\cC^1$, we can derive the following expression for the change of $F$ along $v$:
    \begin{align*}
        F(\overline{x}+\delta v) - F(x) =& f(\overline{x} +\delta v,y(\delta,v)) - f(\overline{x},\overline{y}) \\
        =& \nabla_x f(\overline{x},\overline{y})^T \delta v + \underbrace{\nabla_y f(\overline{x},\overline{y})^T}_{=0} (y(\delta,v)-\overline{y}) + R_1(\delta,v)\\
        =& \nabla_x f(\overline{x},\overline{y})^T \delta v + R_1(\delta,v).
    \end{align*}

    By the fundamental theorem of calculus, the remainder term $R_1(\delta,v)$ is given by:
    \begin{align*}
        R_1(\delta,v) =& \int_0^1  \left(\nabla f(\overline{x}+t \delta v,\overline{y}+t(y(\delta,v)-\overline{y}))-\nabla f(\overline{x},\overline{y})\right)^T \left( \delta v, y(\delta,v)-\overline{y}\right) dt 
    \end{align*}

    Using the Lipschitz continuity of $\nabla f$ on $N^L$ and denoting the Lipschitz constant as $S$, we can bound the remainder term:
    \begin{align*}
        \|R_1(\delta,v)\| \leq & \int_0^1 \|\nabla f(\overline{x}+t \delta v,\overline{y}+t(y(\delta,v)-\overline{y}))-\nabla f(\overline{x},\overline{y})\|  \| \left( \delta v, y(\delta,v)-\overline{y}\right) \| dt \\
        \leq & S \int_0^1 t \|\left( \delta v, y(\delta,v)-\overline{y}\right) \|^2 dt \\
        = & \frac{S}{2} \|\left( \delta v, y(\delta,v)-\overline{y}\right) \|^2 \\
        = & \mathcal{O}(\delta^2 + \delta^{2\alpha_x} )\\
    \end{align*}
    Therefore, independent of the choice of $\overline{y}\in \BR_y(\overline{x})$, we have:
    \begin{align*}
        \lim_{\delta \to 0} \frac{F(\overline{x}+\delta v)-F(\overline{x})}{\delta} =& \nabla_x f(\overline{x},\overline{y})^T v + \lim_{\delta \to 0} \mathcal{O}(\delta+\delta^{2\alpha_x-1})\\
    \end{align*}    
    By Assumption $\alpha_x>0.5$, which implies that the extra term vanishes as $\delta \to 0$. Note, that this result holds for all $\overline{y}\in \BR_y(\overline{x})$. Therefore, all directional derivatives exist, independent of the choice of $\overline{y}\in \BR_y(\overline{x})$, and are continuous functions because $f$ is $\cC^1$. $F$ is therefore differentiable at $\overline{x}$ and its gradient is given by:
    \[
    \nabla F(\overline{x}) = \nabla_x f(\overline{x},\overline{y}),
    \]
    for any $\overline{y} \in \BR_y(\overline{x})$.\footnote{Indeed, from our analysis it follows that $\forall y \in \BR_y(\overline{x})\; \forall \overline{y} \in \BR_y(\overline{x}):\nabla_x f(\overline{x},y)=\nabla_x f(\overline{x},\overline{y})$.} 

    To conclude our proof, consider an arbitrary point $(\overline{x},\overline{y})\in \underline{M}$. By definition, $\nabla_y f(\overline{x},\overline{y})=0$. As $\overline{x}\in \arg\min_x F(x)$, it follows that $\nabla F(\overline{x})=0$ and therefore, using our result above, $\nabla_x f(\overline{x},\overline{y})=\nabla F(\overline{x})=0$. Due to the invex-incave structure of $f$, this implies $(\overline{x},\overline{y})\in E$ and thus $\underline{M}\subseteq E$.
\end{proof}

Again, we are interested in which properties of $f$ ensure the local inner Hölder continuity of $\BR_y(x),\BR_x(y)$ and thus the connectedness of $E=\underline{M}=\overline{M}$. 
In \Cref{prop:PLconnected}, we presented the two-sided PL condition, which captures an important class of nonconvex-nonconcave optimization problems. If we assume $f$ to have locally Lipschitz continuous gradients, then we can replace the assumption that $f(x,\cdot)$ and $f(\cdot,y)$ satisfy the $2$-PL condition with the $\alpha$-PL condition for $\alpha>1.5$. 

Before we state our assumption in detail, we introduce the $\alpha$-growth condition.
\begin{definition}[Growth Condition]
    \label{def:localgrowth}
    Let $f: \R^n \to \R$ be a continuously differentiable function. We say that $f$ has $\alpha$-growth, if there exists a constant $\eta>0$ such that:
    \begin{equation*}
        \forall x\in \R^n: (f(x)-\min_x f(x)) \geq \eta d(x,P)^\alpha,
    \end{equation*}
    where $P=\arg\min f(x)$ is the set of minimizers of $f$.
\end{definition}

\Cref{lem:PLgrowth} in \Cref{app:proofs} shows that a function locally satisfying the $\alpha$-PL condition has $\frac{\alpha}{\alpha-1}$-growth on a possibly smaller neighborhood. Therefore, if we assume the $\alpha$-PL condition we do not need to additionally assume $\frac{\alpha}{\alpha-1}$-growth. However, we explicitly state $\frac{\alpha}{\alpha-1}$-growth in our assumption below to emphasize that the neighborhoods where these conditions hold cannot become arbitrarily small, as we vary $x$ and $y$.

\begin{assum}
   \label{assum:localPL}
    Let $f: \R^{n_x}\times \R^{n_y}\to \R$ be continuously differentiable.
    We assume that:
    \begin{enumerate}
      
        \item $\nabla f$ is locally Lipschitz continuous on $\underline{M}\cup\overline{M}$, i.e. for all $(\overline{x},\overline{y})\in \underline{M}\cup \overline{M}$, there exists a neighborhood $N^L$ containing $(\overline{x},\overline{y})$ and a constant $S$, such that $\nabla f$ is $S$-Lipschitz continuous on $N^L$.
        \item $f$ locally satisfies the $\alpha$-PL and $\frac{\alpha}{\alpha-1}$-growth condition around $X\times Y$, i.e. there exist constants $\alpha_x>1.5,\eta_x,\mu_x$, such that for all $\overline{x}\in X$, there exists a neighborhood $N_{\overline{x}}^\text{PLG}$ and constant $\delta_{\overline{x}}^{(y)}$, such that for all $x' \in N_{\overline{x}}^\text{PLG}$ and all $y' \in \B_{\delta_{\overline{x}}^{(y)}}(\BR(x'))$, the following two conditions hold:
        \begin{align*}
            \max_{y} f(x',y) - f(x',y') &\geq \eta_x d(y',\BR(x'))^{\frac{\alpha_x}{\alpha_x-1}},\\
            \|\nabla_y f(x',y')\|^{\alpha_x} &\geq \mu_x \left( \max_{y} f(x',y) - f(x',y')\right)
        \end{align*}
        Analogously, there exist constants $\alpha_y>1.5,\eta_y,\mu_y$, such that for all $\overline{y}\in Y$, there exists a neighborhood $N_{\overline{y}}^\text{PLG}$ and constant $\delta_{\overline{y}}^{(x)}$, such that for all $y' \in N_{\overline{y}}^\text{PLG}$ and all $x' \in \B_{\delta_{\overline{y}}^{(x)}}(\BR(y'))$, it holds that:\(
            f(x',y') - \min_{x} f(x,y') \geq \eta_y d(x',\BR(y'))^{\frac{\alpha_y}{\alpha_y-1}}\) and 
           \( \|\nabla_x f(x',y')\|^{\alpha_y} \geq \mu_y \left( f(x',y') - \min_{x} f(x,y')\right)\).

    \end{enumerate}
\end{assum}

\begin{prop}
    \label{prop:localPLconnected}
    Let $f$ satisfy \Cref{assum:localPL}. Then the map $\BR_y(x)$ is locally inner $(\alpha_x-1)$-Hölder continuous at $\overline{x}$ for $\overline{y}$ for all $(\overline{x},\overline{y})\in \underline{M}$ and conversely $\BR_x(y)$ is locally inner $(\alpha_y-1)$-Hölder continuous at $\overline{y}$ for $\overline{x}$ for all $(\overline{x},\overline{y})\in \overline{M}$. Therefore, $E=\underline{M}=\overline{M}$ and this common set is connected.
\end{prop}

\begin{proof}
     Connectedness follows from the local Hölder continuity of $\BR_y(x)$ and $\BR_x(y)$ by \Cref{thm:Holderconnected}. We restrict to showing the local inner Hölder continuity of $\BR_y(x)$. The proof for $\BR_x(y)$ follows analogously.
     
     Let $(\overline{x},\overline{y})\in \underline{M}$. From \Cref{lem:isc}, we know that $\BR_y(x)$ is inner semicontinuous at $\overline{x}$, i.e.
     \[
     \forall \epsilon\; \exists \delta>0\; \forall x\in \B_{\delta}(\overline{x}): \BR_y(x)\cap \B_{\epsilon}(\overline{y})\neq \emptyset,
     \]
     which is equivalent to:
        \[
        \forall \epsilon\; \exists \delta>0\; \forall x\in \B_{\delta}(\overline{x}): d(\overline{y},\BR_y(x))\leq \epsilon.
        \]
    In particular, we conclude that there exists a small enough $\delta>0$ such that (i) $\B_\delta(\overline{x})\subseteq N_{\overline{x}}^\text{PLG}\cap N_{\overline{x}}^\text{L}$, and (ii) for all $ x\in \B_{\delta}(\overline{x}): d(\overline{y},\BR_y(x))\leq \delta_{\overline{x}}^{(y)}$, where $\delta_{\overline{x}}^{(y)}$ is the constant from \Cref{assum:localPL} and $N_{\overline{x}}^\text{L}$ is the neighborhood where $\nabla f$ is Lipschitz continuous with constant $S$. Because, we can ensure that $d(\overline{y},\BR_y(x))\leq  \delta_{\overline{x}}^{(y)} $, we can apply the local $\alpha_x$-PL and $\frac{\alpha_x}{\alpha_x-1}$-growth condition to obtain the following statement for all $x\in \B_{\delta}(\overline{x})$:
    \begin{align*}
        S^{\alpha_x}\|x - \bar{x}\|^{\alpha_x} \ \ &\underset{\mathclap{\text{(Lip.)}}}{\geq}\ \  \ \|\nabla_y f(x,\bar{y}) - \nabla_y f(\bar{x},\bar{y})\|^{\alpha_x}\\
        &\underset{\mathclap{\text{(invex)}}}{=} \ \ \ \|\nabla_y f(x,\bar{y})\|^{\alpha_x}\\
        &\underset{\mathclap{\text{(PL)}}}{\geq}  \ \ \  \mu_{x} \left(\max_{y \in Y} f(x,y) - f(x,\bar{y})\right)\\
        &\underset{\mathclap{\text{(growth)}}}{\geq} \ \  \ \eta_x \mu_{x} d(\bar{y}, \BR_y(x))^{\frac{\alpha_x}{\alpha_x-1}}.
    \end{align*}
    Rearranging the terms yields the claimed Hölder continuity: 
    \[
    \forall x \in \B_{\delta}(\overline{x}): d(\overline{y},\BR_y(x)) \leq \kappa \|x-\overline{x}\|^{\alpha_x-1},
    \]
    for $\kappa = \left(S^{\alpha_x-1} \eta_x^{\frac{1-\alpha_x}{\alpha_x}}\mu_x^{\frac{1-\alpha_x}{\alpha_x}}\right)$.
\end{proof}

\begin{example}
    \label{ex:localPLconnected}
    For $x,y \in \R$, define $a = \max\left(|x| - 1, 0\right)$ and $b = \max\left(|y| - 1, 0\right)$ and consider the function 
    \[
        f(x,y)= a^2 \exp\left(-b^2\right) - b^2,
    \]
    plotted in \Cref{fig:3d_plot_e}. This function is nonconcave in $y$, but satisfies the assumptions of \Cref{prop:localPLconnected} (verified in \Cref{app:proofs}). Therefore $E=\underline{M}=\overline{M}$ and this common set is connected. A direct computation shows $E=[-1,1]\times[-1,1]$.

\end{example}

\begin{figure}
    \centering
    \includegraphics[width=0.75\textwidth,trim=0 0.6cm 0 1.8cm,clip]{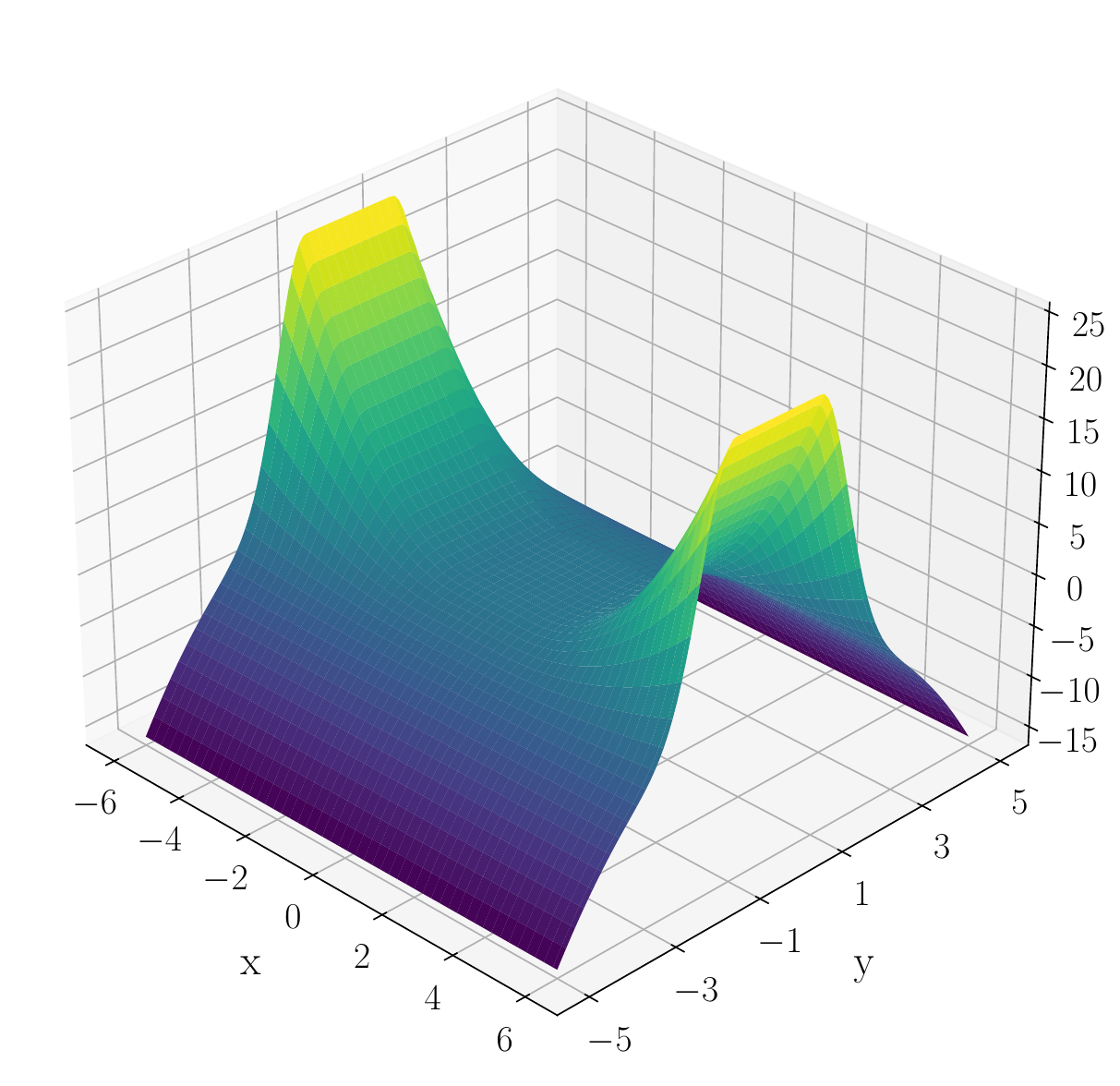}
    \caption{Plot of the function $ f(x,y)= a^2 \exp\left(-b^2\right) - b^2$ with $E=[-1,1]\times[-1,1]$, from \Cref{ex:localPLconnected}. }
    \label{fig:3d_plot_e}
\end{figure}